\documentclass[a4paper,10pt,leqno,openright,oneside]{amsart}
\usepackage[english]{babel}

\usepackage{amsthm,amssymb,amsmath,amsopn,amsfonts,pb-diagram,footmisc,accents}
\usepackage{cite}
\usepackage{stmaryrd,calc,enumerate,mathrsfs,amscd,wasysym,footnote}
\usepackage{graphicx}
\usepackage{fancyhdr}
\usepackage{indentfirst}
\usepackage{geometry}
\usepackage{mathtools}
\usepackage{bbm}
\usepackage{float}
\usepackage{dirtytalk} %per usare comando \say{} per il virgolettato
\usepackage{hyperref}
\usepackage{xcolor}
\usepackage{comment}

\usepackage{pgf,tikz}
\usepackage{mathrsfs}
\usetikzlibrary{arrows}

\usepackage{multicol}

\DeclareMathOperator{\divergence}{div}

\theoremstyle{plain}
\newtheorem{theorem}{Theorem}[section]
\newtheorem{cor}[theorem]{Corollary}

\newtheorem{prop}[theorem]{Proposition}

\newtheorem{thm}{Theorem} % Teorema numerat

\theoremstyle{definition} 

\newtheorem{exmpl}{Example}

\newtheorem{thmA}{Theorem}
\newcommand*{\bigchi}{\mbox{\Large$\chi$}}

\DeclareMathOperator{\supp}{supp}
\DeclareMathOperator{\inn}{int}
\DeclareMathOperator{\sgn}{sgn}
\DeclareMathOperator{\card}{card}

\title{Equilibrium measures on trees}
\author{{Nicola Arcozzi}, {Matteo Levi}}
\thanks{The two authors are partially supported by the 2015 PRIN grant \textit{Real and Complex Manifolds: Geometry, Topology
and Harmonic Analysis} of the Italian Ministry of Education (MIUR)}
\date{}

\subjclass[2010]{Primary: 31C15. Secondary: 05C63, 05C05, 05B45, 52C20}
\keywords{infinite trees; nonlinear potential theory; capacity; equilibrium measures; tilings}

\address{N. Arcozzi \\ Dipartimento di Matematica \\ Universit\'a di Bologna \\ 40127 Bologna, Italy}
\email{nicola.arcozzi@unibo.it}

\address{M. Levi \\ Dipartimento di Matematica \\ Politecnico di Torino \\ 10129 Torino, Italy}
\email{matteo.levi@polito.it}

\begin{document}

\maketitle

\begin{abstract}
We give a characterization of equilibrium measures for $p$-capacities on the boundary of an infinite tree of arbitrary (finite) local degree. 
For $p=2$, this provides, in the special case of trees, a converse to a theorem of Benjamini and Schramm, which interpretes 
the equilibrium measure of a planar graph's boundary in terms of square tilings of cylinders.
\end{abstract}

\section*{Introduction}

In Electrostatics, an amount of, say positive, electric charge free to move across a conductor $A$ in the Euclidean space will reach an equilibrium configuration $\mu$, which at the same time: (i) minimizes the energy ${\mathcal E}(\mu)$ carried by the generated electrostatic potential; (ii) minimizes the maximum value of the potential ; (iii) makes the potential constant on all of $A$, but possibly for a small exceptional set. For a given system of units, there is an amount $\|\mu\|$ of charge for which the potential on (most of) $A$ is unitary. The total charge $\|\mu\|$ is the \textit{capacity} of the conductor and $\mu$ is the corresponding \textit{equilibrium measure} of $A$. The mathematical theory of electrostatics, developed by Gauss, then put on firm mathematical foundations by Frostman, was later extended in many directions. See \cite{brelot} for a survey of axiomatic linear theories which goes far beyond the scope of this article, and \cite{adams} for a rather general axiomatic non-linear theory. The problem we consider here, in a special instance, is that of characterizing equilibrium measures. Namely, given a positive measure $\mu$, our ``measurable'', is there a way to tell whether or not it is the equilibrium measure for some conductor $A$? The equilibrium measures are known to satisfy a number of properties, but to the best of our knowledge a complete answer is available only for the equilibrium measure of the boundary of a planar graphs. The case of finite, planar graphs is in \cite{schramm}, and a combinatorial interpretation of the equilibrium measure $\mu$ of the boundary of possibly infinite, planar graphs satisfying an extra technical hypothesis, is in a beautiful article by Benjamini and Schramm \cite{benjamini}.

In this article, we characterize the equilibrium measures, within Nonlinear Potential Theory, for subsets of the boundary of an infinite tree. Benjamini and Schramm's theorem, in the case of trees, would apply to the equilibrium measures of a closed subset of the tree boundary. We find a condition which characterizes the equilibrium measures of subset of the trees, providing, in this special context, a converse to their result.

In order to state our main finding, we fix some minimal notation, to be better developed in the next section.
A tree $T$ is a connected graph with no cycles. We denote by $E$ the set of its edges and by $V$ the set of its vertices. Given two vertices $x,y$, we write $x\sim y$ if they are connected by an edge. We consider infinite trees such that each vertex has a finite but arbitrary number of neighborhoods. We assume that there is a unique, distinguished edge $\omega$ one of whose endpoints, $o$, is not endpoint of any other edge. We say that $o$ is the \it root \rm of $T$.
 The boundary $\partial T$ of the tree can be classically identified with the set of labels corresponding to half infinite geodesics starting at the root (this coincides with both the Carath\'eodory and the Martin boundary of $T$). We set $\overline{V}=V\cup\partial T$.
We partially order $\overline{V}$ by writing $\xi\geq y$ if and only if $y$ lies on the geodesic connecting $o$ and $\xi$ (or if $y$ belongs to the geodesic labelled by $\xi$ in the case $\xi\in\partial T$). Thus, edges are oriented and we can identify $E$ with the subset of couples $(x,y)$ in $V\times V$ such that $x\sim y$ and $y>x$.  The boundary of the tree is a compact metric space with respect to the \it visual metric \rm (see the section on preliminaries). We have then Borel signed measures on $\partial T$, which we call \textit{charges}, while \textit{measures} are intended to be nonnegative. As we will detail in the next section, using these ingredients and following the classical theory presented in \cite{adams}, we can develop a Nonlinear Potential Theory on the tree. In particular, for any $p\in (1,+\infty)$ one can associate to a charge $\mu$ a nonlinear potential $V_p\mu:\overline{V}\to \mathbb{R}$, and a notion of $p-$capacity for subsets of the boundary is made available. By the general theory, we know that for each \textit{capacitable} $A\subseteq\partial T$, there exists a unique positive measure $\mu^A$ on $\partial T$ realizing its $p-$capacity, namely such that $c_p(A)=\|\mu^A\|$. We call $\mu^A$ the $p-$\textit{equilibrium measure} for $A$.

Equilibrium measures are strictly related to trace inequalities for discrete Hardy operators \cite[Theorem 5]{arcozzi2008capacity}.
The simplest interpretation, however, is in terms of elementary rescaling properties of trees, see Section \ref{SecRescaling}. Our goal is characterizing which measures $\mu$ are equilibrium measures of some subset of the boundary.
Such a characterization is encoded in the following integro-differential equation, which we will call, accordingly, the \textit{equilibrium equation}:
\begin{equation}\label{equilibrium equation}
\nabla g[x,y]|\nabla g|^{p-2}[x,y]\Big(1-g(x)\Big)=\sum_{\substack{[z,w]\in E \\ z\geq x}}|\nabla g[z,w]|^{p}, \quad  [x,y]\in E.
\end{equation}
In the above expression $p,p'\in(1,\infty)$ are H\"older conjugate exponents, $E$ denotes the set of edges of the tree, $g$ is a function defined on the vertices, $\nabla g[x,y]:=g(y)-g(x)$ denotes its gradient on the edge $[x,y]$ and the right-hand side is a Sobolev energy restricted to a the full subtree rooted at $[x,y]$. Our main result is the following.
\begin{thm}\label{main 1}
(i) Let $\mu$ be the $p-$equilibrium measure for a set $A\subseteq \partial T$. Then the function $g=V_p\mu$ solves \eqref{equilibrium equation}.

(ii) Conversely, let $g:V\to\mathbb{R}$ be a solution of \eqref{equilibrium equation} with $\Vert g\Vert_\infty<1$. Then, there exists an $\mathcal{F}_{\sigma}$ set $A\subseteq\partial T$ with equilibrium measure $\mu^A$ such that $g=V_p(\mu^A)$.
\end{thm}
Observe that equation \eqref{equilibrium equation} is non linear even in Linear Potential Theory. This is not surprising, since linear combinations of equilibrium measures are only seldom equilibrium measures themselves. From the discussion in Section \ref{SecFlows} it will be clear that solutions of \eqref{equilibrium equation} are automatically $p-$harmonic functions which are increasing along geodesics emanating from $o$. We will see how the equilibrium equation can be regarded as an equation for measures and for edge functions, leading to the reformulations \eqref{eq:formula} and \eqref{eq:edges} respectively.

The equation can be interpreted in several ways. In the linear case $p=2$  it says that equilibrium measures on trees can be associated to particular tilings of rectangles by squares. This gives an independent proof of the tiling theorem of Benjamini and Schramm \cite[Theorem 4.1]{benjamini}, in the special case of a tree, and, more interestingly, it provides a converse result.

Here is the precise statement (see Section \ref{SecTiling}). We say that a square tiling of the rectangle $R=[0,t]\times[0,1]$ has combinatorics prescribed by a tree $T$ if there is a bijection $\alpha\mapsto Q_\alpha$ from the edge set of $T$ to the set of tiles such that $\alpha$ and $\beta$ have a common vertex and $\alpha<\beta$ if and only if the squares $Q_\alpha$ and $Q_\beta$ are 	neighbouring in the tiling and the upper side of $Q_\beta$ lies on the lower side of $Q_\alpha$.

\begin{multicols}{2}

\definecolor{qqqqff}{rgb}{0.,0.,1.}
\definecolor{ccqqqq}{rgb}{0.8,0.,0.}
\begin{tikzpicture}[line cap=round,line join=round,>=triangle 45,x=1.0cm,y=1.0cm]
\clip(-1.,-0.5) rectangle (5.,8.5);
\fill[line width=0.8pt, fill opacity=0] (4.,8.) -- (0.,8.) -- (0.,4.) -- (4.,4.) -- cycle;
\fill[line width=0.8pt, fill opacity=0] (4.,4.) -- (1.8490020792533923,4.) -- (1.849002079253392,1.8490020792533943) -- (4.,1.849002079253393) -- cycle;
\fill[line width=0.8pt, fill opacity=0] (1.8490020792533923,4.) -- (0.,4.) -- (0.,2.1509979207466072) -- (1.849002079253392,2.1509979207466072) -- cycle;
\fill[line width=0.8pt, fill opacity=0] (1.849002079253392,2.1509979207466072) -- (1.3134632751944144,2.1509979207466072) -- (1.3134632751944142,1.6154591166876298) -- (1.8490020792533919,1.6154591166876295) -- cycle;
\fill[line width=0.8pt, fill opacity=0] (4.,1.849002079253393) -- (3.5419723094289957,1.8490020792533932) -- (3.5419723094289957,1.3909743886823909) -- (4.,1.3909743886823907) -- cycle;
\fill[line width=0.8pt, fill opacity=0] (3.5419723094289957,1.8490020792533932) -- (2.6218324672246793,1.8490020792533939) -- (2.621832467224679,0.9288622370490778) -- (3.5419723094289948,0.928862237049077) -- cycle;
\fill[line width=0.8pt, fill opacity=0] (2.6218324672246793,1.8490020792533939) -- (1.849002079253392,1.8490020792533943) -- (1.8490020792533923,1.0761716912821073) -- (2.621832467224679,1.076171691282107) -- cycle;
\fill[line width=0.8pt, fill opacity=0] (1.8490020792533919,1.6154591166876295) -- (1.3134632751944142,1.6154591166876298) -- (1.3134632751944144,1.0799203126286525) -- (1.8490020792533917,1.0799203126286523) -- cycle;
\fill[line width=0.8pt, fill opacity=0] (1.8490020792533917,1.0799203126286523) -- (1.616297016381219,1.0799203126286523) -- (1.6162970163812187,0.8472152497564795) -- (1.8490020792533914,0.8472152497564794) -- cycle;
\fill[line width=0.8pt, fill opacity=0] (1.616297016381219,1.0799203126286523) -- (1.3134632751944144,1.0799203126286525) -- (1.3134632751944144,0.7770865714418478) -- (1.616297016381219,0.7770865714418478) -- cycle;
\fill[line width=0.8pt, fill opacity=0] (1.8490020792533914,0.8472152497564794) -- (1.6162970163812187,0.8472152497564795) -- (1.616297016381219,0.6145101868843069) -- (1.8490020792533914,0.6145101868843069) -- cycle;
\fill[line width=0.8pt, fill opacity=0] (1.616297016381219,0.7770865714418478) -- (1.4892160113937543,0.7770865714418478) -- (1.4892160113937543,0.6500055664543831) -- (1.616297016381219,0.6500055664543831) -- cycle;
\fill[line width=0.8pt, fill opacity=0] (1.4892160113937543,0.7770865714418478) -- (1.3134632751944144,0.7770865714418478) -- (1.3134632751944144,0.6013338352425079) -- (1.4892160113937543,0.6013338352425079) -- cycle;
\fill[line width=0.8pt, fill opacity=0] (4.,1.3909743886823907) -- (3.5419723094289957,1.3909743886823909) -- (3.541972309428996,0.9329466981113885) -- (4.,0.9329466981113888) -- cycle;
\fill[line width=0.8pt, fill opacity=0] (3.5419723094289948,0.928862237049077) -- (3.2438189126237997,0.9288622370490772) -- (3.2438189126237997,0.6307088402438822) -- (3.5419723094289948,0.6307088402438821) -- cycle;
\fill[line width=0.8pt, fill opacity=0] (3.2438189126237997,0.9288622370490772) -- (2.943013355541281,0.9288622370490776) -- (2.9430133555412805,0.6280566799665587) -- (3.2438189126237993,0.6280566799665583) -- cycle;
\fill[line width=0.8pt, fill opacity=0] (2.943013355541281,0.9288622370490776) -- (2.825017921339196,0.9288622370490776) -- (2.8250179213391955,0.8108668028469925) -- (2.9430133555412805,0.8108668028469923) -- cycle;
\fill[line width=0.8pt, fill opacity=0] (2.825017921339196,0.9288622370490776) -- (2.621832467224679,0.9288622370490778) -- (2.621832467224679,0.7256767829345607) -- (2.825017921339196,0.7256767829345606) -- cycle;
\fill[line width=0.8pt, fill opacity=0] (4.,0.9329466981113888) -- (3.7451577635435136,0.9329466981113886) -- (3.745157763543514,0.6781044616549046) -- (4.,0.6781044616549048) -- cycle;
\fill[line width=0.8pt, fill opacity=0] (3.7451577635435136,0.9329466981113886) -- (3.541972309428996,0.9329466981113885) -- (3.5419723094289957,0.7297612439968709) -- (3.7451577635435136,0.7297612439968708) -- cycle;
\fill[line width=0.8pt, fill opacity=0] (2.621832467224679,1.076171691282107) -- (2.2129381466058966,1.0761716912821073) -- (2.2129381466058966,0.6672773706633248) -- (2.621832467224679,0.6672773706633248) -- cycle;
\fill[line width=0.8pt, fill opacity=0] (2.2129381466058966,1.0761716912821073) -- (1.8490020792533923,1.0761716912821073) -- (1.8490020792533919,0.712235623929603) -- (2.212938146605896,0.7122356239296028) -- cycle;
\fill[line width=0.8pt, fill opacity=0] (1.616297016381219,0.6500055664543831) -- (1.4892160113937543,0.6500055664543831) -- (1.4892160113937543,0.5229245614669185) -- (1.616297016381219,0.5229245614669185) -- cycle;
\fill[line width=0.8pt, fill opacity=0] (1.8490020792533914,0.6145101868843069) -- (1.7285497690226117,0.6145101868843069) -- (1.7285497690226115,0.4940578766535272) -- (1.8490020792533912,0.4940578766535271) -- cycle;
\fill[line width=0.8pt, fill opacity=0] (1.7285497690226117,0.6145101868843069) -- (1.616297016381219,0.6145101868843069) -- (1.6162970163812187,0.5022574342429141) -- (1.7285497690226115,0.502257434242914) -- cycle;
\fill[line width=0.8pt, fill opacity=0] (1.4892160113937543,0.6013338352425079) -- (1.410286793452321,0.6013338352425079) -- (1.4102867934523209,0.5224046173010748) -- (1.489216011393754,0.5224046173010747) -- cycle;
\fill[line width=0.8pt, fill opacity=0] (1.410286793452321,0.6013338352425079) -- (1.3134632751944144,0.6013338352425079) -- (1.3134632751944142,0.5045103169846012) -- (1.4102867934523209,0.5045103169846011) -- cycle;
\fill[line width=0.8pt, fill opacity=0] (2.2129381466058966,0.7122356239296029) -- (2.035385045419575,0.7122356239296029) -- (2.0353850454195754,0.5346825227432812) -- (2.212938146605897,0.5346825227432817) -- cycle;
\fill[line width=0.8pt, fill opacity=0] (2.035385045419575,0.7122356239296029) -- (1.8490020792533919,0.7122356239296033) -- (1.8490020792533914,0.52585265776342) -- (2.0353850454195745,0.5258526577634197) -- cycle;
\fill[line width=0.8pt, fill opacity=0] (2.6218324672246784,0.6672773706633245) -- (2.490019429773477,0.6672773706633248) -- (2.4900194297734766,0.5354643332121234) -- (2.621832467224678,0.535464333212123) -- cycle;
\fill[line width=0.8pt, fill opacity=0] (2.490019429773477,0.6672773706633248) -- (2.3508338750390565,0.6672773706633248) -- (2.3508338750390565,0.5280918159289043) -- (2.490019429773477,0.5280918159289043) -- cycle;
\fill[line width=0.8pt, fill opacity=0] (2.3508338750390565,0.6672773706633248) -- (2.212938146605896,0.6672773706633245) -- (2.212938146605896,0.529381642230164) -- (2.3508338750390565,0.5293816422301643) -- cycle;
\fill[line width=0.8pt, fill opacity=0] (2.8250179213391977,0.7256767829345606) -- (2.621832467224678,0.7256767829345608) -- (2.621832467224678,0.5224913288200411) -- (2.8250179213391977,0.5224913288200411) -- cycle;
\fill[line width=0.8pt, fill opacity=0] (2.943013355541281,0.810866802846994) -- (2.8250179213391946,0.8108668028469942) -- (2.825017921339195,0.6928713686449077) -- (2.943013355541281,0.692871368644908) -- cycle;
\fill[line width=0.8pt, fill opacity=0] (2.943013355541281,0.692871368644908) -- (2.825017921339195,0.6928713686449077) -- (2.8250179213391955,0.5748759344428218) -- (2.943013355541281,0.5748759344428223) -- cycle;
\fill[line width=0.8pt, fill opacity=0] (3.2438189126237993,0.6280566799665583) -- (3.1418629537083134,0.6280566799665585) -- (3.1418629537083134,0.5261007210510723) -- (3.2438189126237993,0.5261007210510723) -- cycle;
\fill[line width=0.8pt, fill opacity=0] (3.1418629537083134,0.6280566799665585) -- (3.0460861438180094,0.6280566799665586) -- (3.0460861438180094,0.5322798700762547) -- (3.1418629537083134,0.5322798700762545) -- cycle;
\fill[line width=0.8pt, fill opacity=0] (3.0460861438180094,0.6280566799665586) -- (2.943013355541281,0.6280566799665588) -- (2.943013355541281,0.5249838916898304) -- (3.0460861438180094,0.5249838916898303) -- cycle;
\fill[line width=0.8pt, fill opacity=0] (2.943013355541281,0.5748759344428223) -- (2.8859463279440996,0.574875934442822) -- (2.8859463279440996,0.5178089068456407) -- (2.943013355541281,0.5178089068456407) -- cycle;
\fill[line width=0.8pt, fill opacity=0] (2.8859463279440996,0.574875934442822) -- (2.8250179213391955,0.5748759344428218) -- (2.8250179213391955,0.5139475278379178) -- (2.8859463279440996,0.5139475278379179) -- cycle;
\fill[line width=0.8pt, fill opacity=0] (3.5419723094289943,0.6307088402438834) -- (3.4390755109878297,0.6307088402438821) -- (3.4390755109878315,0.5278120418027176) -- (3.5419723094289957,0.527812041802719) -- cycle;
\fill[line width=0.8pt, fill opacity=0] (3.4390755109878297,0.6307088402438821) -- (3.340616716797701,0.6307088402438822) -- (3.340616716797701,0.5322500460537534) -- (3.4390755109878297,0.5322500460537534) -- cycle;
\fill[line width=0.8pt, fill opacity=0] (3.340616716797701,0.6307088402438822) -- (3.2438189126238,0.6307088402438836) -- (3.243818912623798,0.5339110360699829) -- (3.3406167167976992,0.533911036069981) -- cycle;
\fill[line width=0.8pt, fill opacity=0] (4.,0.678104461654905) -- (3.8641663866677303,0.6781044616549047) -- (3.864166386667731,0.5422708483226394) -- (4.,0.5422708483226398) -- cycle;
\fill[line width=0.8pt, fill opacity=0] (3.8641663866677303,0.6781044616549047) -- (3.745157763543514,0.6781044616549048) -- (3.745157763543514,0.5590958385306887) -- (3.8641663866677303,0.5590958385306886) -- cycle;
\fill[line width=0.8pt, fill opacity=0] (3.745157763543515,0.7297612439968706) -- (3.644100369671894,0.7297612439968708) -- (3.644100369671894,0.6287038501252499) -- (3.745157763543515,0.6287038501252498) -- cycle;
\fill[line width=0.8pt, fill opacity=0] (3.644100369671894,0.7297612439968708) -- (3.5419723094289948,0.7312034070612733) -- (3.540530146364592,0.6290753468183741) -- (3.642658206607491,0.6276331837539715) -- cycle;
\fill[line width=0.8pt, fill opacity=0] (3.642658206607491,0.6276331837539715) -- (3.5915095592080015,0.6283554601772997) -- (3.5907872827846727,0.5772068127778103) -- (3.6419359301841623,0.5764845363544817) -- cycle;
\fill[line width=0.8pt, fill opacity=0] (3.5915095592080015,0.6283554601772997) -- (3.540530146364592,0.6290753468183741) -- (3.539810259723517,0.5780959339749644) -- (3.590789672566927,0.5773760473338898) -- cycle;
\fill[line width=0.8pt, fill opacity=0] (3.745157763543515,0.6287038501252498) -- (3.6938550656876656,0.6287038501252499) -- (3.693855065687666,0.5774011522694005) -- (3.745157763543515,0.5774011522694007) -- cycle;
\fill[line width=0.8pt, fill opacity=0] (3.6938550656876656,0.6287038501252499) -- (3.644100369671894,0.6287038501252499) -- (3.6441003696718934,0.5789491541094781) -- (3.693855065687665,0.5789491541094779) -- cycle;
\fill[line width=0.8pt, fill opacity=0] (3.8641663866677303,0.5590958385306886) -- (3.8107596057251354,0.5590958385306886) -- (3.8107596057251354,0.5056890575880937) -- (3.8641663866677303,0.5056890575880937) -- cycle;
\fill[line width=0.8pt, fill opacity=0] (3.8107596057251354,0.5590958385306886) -- (3.745157763543514,0.5590958385306887) -- (3.745157763543514,0.4934939963490672) -- (3.8107596057251354,0.4934939963490672) -- cycle;
\fill[line width=0.8pt, fill opacity=0] (3.745157763543515,0.5774011522694007) -- (3.693855065687666,0.5774011522694005) -- (3.693855065687667,0.5260984544135516) -- (3.7451577635435154,0.5260984544135523) -- cycle;
\fill[line width=0.8pt, fill opacity=0] (4.,0.5422708483226398) -- (3.9545031021064716,0.5422708483226397) -- (3.9545031021064707,0.49677395042911576) -- (4.,0.4967739504291153) -- cycle;
\fill[line width=0.8pt, fill opacity=0] (3.9545031021064716,0.5422708483226397) -- (3.909014905504744,0.5422708483226395) -- (3.909014905504744,0.49678265172091196) -- (3.9545031021064716,0.49678265172091207) -- cycle;
\fill[line width=0.8pt, fill opacity=0] (3.909014905504744,0.5422708483226395) -- (3.864166386667731,0.5422708483226394) -- (3.8641663866677316,0.4974223294856266) -- (3.909014905504744,0.4974223294856268) -- cycle;
\fill[line width=0.8pt, fill opacity=0] (3.693855065687665,0.5789491541094779) -- (3.6441003696718934,0.5789491541094781) -- (3.644100369671894,0.5291944580937064) -- (3.693855065687665,0.5291944580937065) -- cycle;
\fill[line width=0.8pt, fill opacity=0] (3.6419359301841623,0.5764845363544817) -- (3.5907872827846727,0.5772068127778103) -- (3.590065006361345,0.5260581653783207) -- (3.641213653760834,0.5253358889549925) -- cycle;
\fill[line width=0.8pt, fill opacity=0] (3.5907872827846727,0.5772068127778103) -- (3.539810259723517,0.5780959339749644) -- (3.538921138526362,0.5271189109138086) -- (3.5898981615875183,0.5262297897166541) -- cycle;
\fill[line width=2.pt,color=black, fill opacity=0] (4.,8.) -- (0.,8.) -- (0.004368420828023068,0.5076994075353054) -- (4.007847529220136,0.501168446510196) -- cycle;
\draw [line width=0.8pt] (4.,8.)-- (0.,8.);
\draw [line width=0.8pt] (0.,8.)-- (0.,4.);
\draw [line width=0.8pt] (0.,4.)-- (4.,4.);
\draw [line width=0.8pt] (4.,4.)-- (4.,8.);
\draw [line width=0.8pt] (4.,4.)-- (1.8490020792533923,4.);
\draw [line width=0.8pt] (1.8490020792533923,4.)-- (1.849002079253392,1.8490020792533943);
\draw [line width=0.8pt] (1.849002079253392,1.8490020792533943)-- (4.,1.849002079253393);
\draw [line width=0.8pt] (4.,1.849002079253393)-- (4.,4.);
\draw [line width=0.8pt] (1.8490020792533923,4.)-- (0.,4.);
\draw [line width=0.8pt] (0.,4.)-- (0.,2.1509979207466072);
\draw [line width=0.8pt] (0.,2.1509979207466072)-- (1.849002079253392,2.1509979207466072);
\draw [line width=0.8pt] (1.849002079253392,2.1509979207466072)-- (1.8490020792533923,4.);
\draw [line width=0.8pt] (1.849002079253392,2.1509979207466072)-- (1.3134632751944144,2.1509979207466072);
\draw [line width=0.8pt] (1.3134632751944144,2.1509979207466072)-- (1.3134632751944142,1.6154591166876298);
\draw [line width=0.8pt] (1.3134632751944142,1.6154591166876298)-- (1.8490020792533919,1.6154591166876295);
\draw [line width=0.8pt] (1.8490020792533919,1.6154591166876295)-- (1.849002079253392,2.1509979207466072);
\draw [line width=0.8pt] (1.3134632751944144,2.1509979207466072)-- (0.,2.1509979207466072);
\draw [line width=0.8pt] (0.,2.1509979207466072)-- (0.,0.8375346455521929);
\draw [line width=0.8pt] (0.,0.8375346455521929)-- (1.313463275194414,0.8375346455521925);
\draw [line width=0.8pt] (1.313463275194414,0.8375346455521925)-- (1.3134632751944144,2.1509979207466072);
\draw [line width=0.8pt] (4.,1.849002079253393)-- (3.5419723094289957,1.8490020792533932);
\draw [line width=0.8pt] (3.5419723094289957,1.8490020792533932)-- (3.5419723094289957,1.3909743886823909);
\draw [line width=0.8pt] (3.5419723094289957,1.3909743886823909)-- (4.,1.3909743886823907);
\draw [line width=0.8pt] (4.,1.3909743886823907)-- (4.,1.849002079253393);
\draw [line width=0.8pt] (3.5419723094289957,1.8490020792533932)-- (2.6218324672246793,1.8490020792533939);
\draw [line width=0.8pt] (2.6218324672246793,1.8490020792533939)-- (2.621832467224679,0.9288622370490778);
\draw [line width=0.8pt] (2.621832467224679,0.9288622370490778)-- (3.5419723094289948,0.928862237049077);
\draw [line width=0.8pt] (3.5419723094289948,0.928862237049077)-- (3.5419723094289957,1.8490020792533932);
\draw [line width=0.8pt] (2.6218324672246793,1.8490020792533939)-- (1.849002079253392,1.8490020792533943);
\draw [line width=0.8pt] (1.849002079253392,1.8490020792533943)-- (1.8490020792533923,1.0761716912821073);
\draw [line width=0.8pt] (1.8490020792533923,1.0761716912821073)-- (2.621832467224679,1.076171691282107);
\draw [line width=0.8pt] (2.621832467224679,1.076171691282107)-- (2.6218324672246793,1.8490020792533939);
\draw [line width=0.8pt] (1.8490020792533919,1.6154591166876295)-- (1.3134632751944142,1.6154591166876298);
\draw [line width=0.8pt] (1.3134632751944142,1.6154591166876298)-- (1.3134632751944144,1.0799203126286525);
\draw [line width=0.8pt] (1.3134632751944144,1.0799203126286525)-- (1.8490020792533917,1.0799203126286523);
\draw [line width=0.8pt] (1.8490020792533917,1.0799203126286523)-- (1.8490020792533919,1.6154591166876295);
\draw [line width=0.8pt] (1.8490020792533917,1.0799203126286523)-- (1.616297016381219,1.0799203126286523);
\draw [line width=0.8pt] (1.616297016381219,1.0799203126286523)-- (1.6162970163812187,0.8472152497564795);
\draw [line width=0.8pt] (1.6162970163812187,0.8472152497564795)-- (1.8490020792533914,0.8472152497564794);
\draw [line width=0.8pt] (1.8490020792533914,0.8472152497564794)-- (1.8490020792533917,1.0799203126286523);
\draw [line width=0.8pt] (1.616297016381219,1.0799203126286523)-- (1.3134632751944144,1.0799203126286525);
\draw [line width=0.8pt] (1.3134632751944144,1.0799203126286525)-- (1.3134632751944144,0.7770865714418478);
\draw [line width=0.8pt] (1.3134632751944144,0.7770865714418478)-- (1.616297016381219,0.7770865714418478);
\draw [line width=0.8pt] (1.616297016381219,0.7770865714418478)-- (1.616297016381219,1.0799203126286523);
\draw [line width=0.8pt] (1.313463275194414,0.8375346455521925)-- (0.9774775928253696,0.8375346455521926);
\draw [line width=0.8pt] (0.9774775928253696,0.8375346455521926)-- (0.9774775928253696,0.5015489631831483);
\draw [line width=0.8pt] (0.9774775928253696,0.5015489631831483)-- (1.3134632751944137,0.5015489631831482);
\draw [line width=0.8pt] (1.3134632751944137,0.5015489631831482)-- (1.313463275194414,0.8375346455521925);
\draw [line width=0.8pt] (0.9774775928253696,0.8375346455521926)-- (0.6497142608492261,0.8375346455521927);
\draw [line width=0.8pt] (0.6497142608492261,0.8375346455521927)-- (0.6497142608492261,0.5097713135760491);
\draw [line width=0.8pt] (0.6497142608492261,0.5097713135760491)-- (0.9774775928253696,0.5097713135760491);
\draw [line width=0.8pt] (0.9774775928253696,0.5097713135760491)-- (0.9774775928253696,0.8375346455521926);
\draw [line width=0.8pt] (0.6497142608492261,0.8375346455521927)-- (0.32850619551260535,0.8375346455521928);
\draw [line width=0.8pt] (0.32850619551260535,0.8375346455521928)-- (0.3285061955126053,0.5163265802155722);
\draw [line width=0.8pt] (0.3285061955126053,0.5163265802155722)-- (0.6497142608492259,0.516326580215572);
\draw [line width=0.8pt] (0.6497142608492259,0.516326580215572)-- (0.6497142608492261,0.8375346455521927);
\draw [line width=0.8pt] (0.32850619551260535,0.8375346455521928)-- (0.,0.8375346455521929);
\draw [line width=0.8pt] (0.,0.8375346455521929)-- (0.,0.509028450039587);
\draw [line width=0.8pt] (0.,0.509028450039587)-- (0.32850619551260524,0.5090284500395869);
\draw [line width=0.8pt] (0.32850619551260524,0.5090284500395869)-- (0.32850619551260535,0.8375346455521928);
\draw [line width=0.8pt] (1.8490020792533914,0.8472152497564794)-- (1.6162970163812187,0.8472152497564795);
\draw [line width=0.8pt] (1.6162970163812187,0.8472152497564795)-- (1.616297016381219,0.6145101868843069);
\draw [line width=0.8pt] (1.616297016381219,0.6145101868843069)-- (1.8490020792533914,0.6145101868843069);
\draw [line width=0.8pt] (1.8490020792533914,0.6145101868843069)-- (1.8490020792533914,0.8472152497564794);
\draw [line width=0.8pt] (1.616297016381219,0.7770865714418478)-- (1.4892160113937543,0.7770865714418478);
\draw [line width=0.8pt] (1.4892160113937543,0.7770865714418478)-- (1.4892160113937543,0.6500055664543831);
\draw [line width=0.8pt] (1.4892160113937543,0.6500055664543831)-- (1.616297016381219,0.6500055664543831);
\draw [line width=0.8pt] (1.616297016381219,0.6500055664543831)-- (1.616297016381219,0.7770865714418478);
\draw [line width=0.8pt] (1.4892160113937543,0.7770865714418478)-- (1.3134632751944144,0.7770865714418478);
\draw [line width=0.8pt] (1.3134632751944144,0.7770865714418478)-- (1.3134632751944144,0.6013338352425079);
\draw [line width=0.8pt] (1.3134632751944144,0.6013338352425079)-- (1.4892160113937543,0.6013338352425079);
\draw [line width=0.8pt] (1.4892160113937543,0.6013338352425079)-- (1.4892160113937543,0.7770865714418478);
\draw [line width=0.8pt] (4.,1.3909743886823907)-- (3.5419723094289957,1.3909743886823909);
\draw [line width=0.8pt] (3.5419723094289957,1.3909743886823909)-- (3.541972309428996,0.9329466981113885);
\draw [line width=0.8pt] (3.541972309428996,0.9329466981113885)-- (4.,0.9329466981113888);
\draw [line width=0.8pt] (4.,0.9329466981113888)-- (4.,1.3909743886823907);
\draw [line width=0.8pt] (3.5419723094289948,0.928862237049077)-- (3.2438189126237997,0.9288622370490772);
\draw [line width=0.8pt] (3.2438189126237997,0.9288622370490772)-- (3.2438189126237997,0.6307088402438822);
\draw [line width=0.8pt] (3.2438189126237997,0.6307088402438822)-- (3.5419723094289948,0.6307088402438821);
\draw [line width=0.8pt] (3.5419723094289948,0.6307088402438821)-- (3.5419723094289948,0.928862237049077);
\draw [line width=0.8pt] (3.2438189126237997,0.9288622370490772)-- (2.943013355541281,0.9288622370490776);
\draw [line width=0.8pt] (2.943013355541281,0.9288622370490776)-- (2.9430133555412805,0.6280566799665587);
\draw [line width=0.8pt] (2.9430133555412805,0.6280566799665587)-- (3.2438189126237993,0.6280566799665583);
\draw [line width=0.8pt] (3.2438189126237993,0.6280566799665583)-- (3.2438189126237997,0.9288622370490772);
\draw [line width=0.8pt] (2.943013355541281,0.9288622370490776)-- (2.825017921339196,0.9288622370490776);
\draw [line width=0.8pt] (2.825017921339196,0.9288622370490776)-- (2.8250179213391955,0.8108668028469925);
\draw [line width=0.8pt] (2.8250179213391955,0.8108668028469925)-- (2.9430133555412805,0.8108668028469923);
\draw [line width=0.8pt] (2.9430133555412805,0.8108668028469923)-- (2.943013355541281,0.9288622370490776);
\draw [line width=0.8pt] (2.825017921339196,0.9288622370490776)-- (2.621832467224679,0.9288622370490778);
\draw [line width=0.8pt] (2.621832467224679,0.9288622370490778)-- (2.621832467224679,0.7256767829345607);
\draw [line width=0.8pt] (2.621832467224679,0.7256767829345607)-- (2.825017921339196,0.7256767829345606);
\draw [line width=0.8pt] (2.825017921339196,0.7256767829345606)-- (2.825017921339196,0.9288622370490776);
\draw [line width=0.8pt] (4.,0.9329466981113888)-- (3.7451577635435136,0.9329466981113886);
\draw [line width=0.8pt] (3.7451577635435136,0.9329466981113886)-- (3.745157763543514,0.6781044616549046);
\draw [line width=0.8pt] (3.745157763543514,0.6781044616549046)-- (4.,0.6781044616549048);
\draw [line width=0.8pt] (4.,0.6781044616549048)-- (4.,0.9329466981113888);
\draw [line width=0.8pt] (3.7451577635435136,0.9329466981113886)-- (3.541972309428996,0.9329466981113885);
\draw [line width=0.8pt] (3.541972309428996,0.9329466981113885)-- (3.5419723094289957,0.7297612439968709);
\draw [line width=0.8pt] (3.5419723094289957,0.7297612439968709)-- (3.7451577635435136,0.7297612439968708);
\draw [line width=0.8pt] (3.7451577635435136,0.7297612439968708)-- (3.7451577635435136,0.9329466981113886);
\draw [line width=0.8pt] (2.621832467224679,1.076171691282107)-- (2.2129381466058966,1.0761716912821073);
\draw [line width=0.8pt] (2.2129381466058966,1.0761716912821073)-- (2.2129381466058966,0.6672773706633248);
\draw [line width=0.8pt] (2.2129381466058966,0.6672773706633248)-- (2.621832467224679,0.6672773706633248);
\draw [line width=0.8pt] (2.621832467224679,0.6672773706633248)-- (2.621832467224679,1.076171691282107);
\draw [line width=0.8pt] (2.2129381466058966,1.0761716912821073)-- (1.8490020792533923,1.0761716912821073);
\draw [line width=0.8pt] (1.8490020792533923,1.0761716912821073)-- (1.8490020792533919,0.712235623929603);
\draw [line width=0.8pt] (1.8490020792533919,0.712235623929603)-- (2.212938146605896,0.7122356239296028);
\draw [line width=0.8pt] (2.212938146605896,0.7122356239296028)-- (2.2129381466058966,1.0761716912821073);
\draw [line width=0.8pt] (1.616297016381219,0.6500055664543831)-- (1.4892160113937543,0.6500055664543831);
\draw [line width=0.8pt] (1.4892160113937543,0.6500055664543831)-- (1.4892160113937543,0.5229245614669185);
\draw [line width=0.8pt] (1.4892160113937543,0.5229245614669185)-- (1.616297016381219,0.5229245614669185);
\draw [line width=0.8pt] (1.616297016381219,0.5229245614669185)-- (1.616297016381219,0.6500055664543831);
\draw [line width=0.8pt] (1.8490020792533914,0.6145101868843069)-- (1.7285497690226117,0.6145101868843069);
\draw [line width=0.8pt] (1.7285497690226117,0.6145101868843069)-- (1.7285497690226115,0.4940578766535272);
\draw [line width=0.8pt] (1.7285497690226115,0.4940578766535272)-- (1.8490020792533912,0.4940578766535271);
\draw [line width=0.8pt] (1.8490020792533912,0.4940578766535271)-- (1.8490020792533914,0.6145101868843069);
\draw [line width=0.8pt] (1.7285497690226117,0.6145101868843069)-- (1.616297016381219,0.6145101868843069);
\draw [line width=0.8pt] (1.616297016381219,0.6145101868843069)-- (1.6162970163812187,0.5022574342429141);
\draw [line width=0.8pt] (1.6162970163812187,0.5022574342429141)-- (1.7285497690226115,0.502257434242914);
\draw [line width=0.8pt] (1.7285497690226115,0.502257434242914)-- (1.7285497690226117,0.6145101868843069);
\draw [line width=0.8pt] (1.4892160113937543,0.6013338352425079)-- (1.410286793452321,0.6013338352425079);
\draw [line width=0.8pt] (1.410286793452321,0.6013338352425079)-- (1.4102867934523209,0.5224046173010748);
\draw [line width=0.8pt] (1.4102867934523209,0.5224046173010748)-- (1.489216011393754,0.5224046173010747);
\draw [line width=0.8pt] (1.489216011393754,0.5224046173010747)-- (1.4892160113937543,0.6013338352425079);
\draw [line width=0.8pt] (1.410286793452321,0.6013338352425079)-- (1.3134632751944144,0.6013338352425079);
\draw [line width=0.8pt] (1.3134632751944144,0.6013338352425079)-- (1.3134632751944142,0.5045103169846012);
\draw [line width=0.8pt] (1.3134632751944142,0.5045103169846012)-- (1.4102867934523209,0.5045103169846011);
\draw [line width=0.8pt] (1.4102867934523209,0.5045103169846011)-- (1.410286793452321,0.6013338352425079);
\draw [line width=0.8pt] (2.2129381466058966,0.7122356239296029)-- (2.035385045419575,0.7122356239296029);
\draw [line width=0.8pt] (2.035385045419575,0.7122356239296029)-- (2.0353850454195754,0.5346825227432812);
\draw [line width=0.8pt] (2.0353850454195754,0.5346825227432812)-- (2.212938146605897,0.5346825227432817);
\draw [line width=0.8pt] (2.212938146605897,0.5346825227432817)-- (2.2129381466058966,0.7122356239296029);
\draw [line width=0.8pt] (2.035385045419575,0.7122356239296029)-- (1.8490020792533919,0.7122356239296033);
\draw [line width=0.8pt] (1.8490020792533919,0.7122356239296033)-- (1.8490020792533914,0.52585265776342);
\draw [line width=0.8pt] (1.8490020792533914,0.52585265776342)-- (2.0353850454195745,0.5258526577634197);
\draw [line width=0.8pt] (2.0353850454195745,0.5258526577634197)-- (2.035385045419575,0.7122356239296029);
\draw [line width=0.8pt] (2.6218324672246784,0.6672773706633245)-- (2.490019429773477,0.6672773706633248);
\draw [line width=0.8pt] (2.490019429773477,0.6672773706633248)-- (2.4900194297734766,0.5354643332121234);
\draw [line width=0.8pt] (2.4900194297734766,0.5354643332121234)-- (2.621832467224678,0.535464333212123);
\draw [line width=0.8pt] (2.621832467224678,0.535464333212123)-- (2.6218324672246784,0.6672773706633245);
\draw [line width=0.8pt] (2.490019429773477,0.6672773706633248)-- (2.3508338750390565,0.6672773706633248);
\draw [line width=0.8pt] (2.3508338750390565,0.6672773706633248)-- (2.3508338750390565,0.5280918159289043);
\draw [line width=0.8pt] (2.3508338750390565,0.5280918159289043)-- (2.490019429773477,0.5280918159289043);
\draw [line width=0.8pt] (2.490019429773477,0.5280918159289043)-- (2.490019429773477,0.6672773706633248);
\draw [line width=0.8pt] (2.3508338750390565,0.6672773706633248)-- (2.212938146605896,0.6672773706633245);
\draw [line width=0.8pt] (2.212938146605896,0.6672773706633245)-- (2.212938146605896,0.529381642230164);
\draw [line width=0.8pt] (2.212938146605896,0.529381642230164)-- (2.3508338750390565,0.5293816422301643);
\draw [line width=0.8pt] (2.3508338750390565,0.5293816422301643)-- (2.3508338750390565,0.6672773706633248);
\draw [line width=0.8pt] (2.8250179213391977,0.7256767829345606)-- (2.621832467224678,0.7256767829345608);
\draw [line width=0.8pt] (2.621832467224678,0.7256767829345608)-- (2.621832467224678,0.5224913288200411);
\draw [line width=0.8pt] (2.621832467224678,0.5224913288200411)-- (2.8250179213391977,0.5224913288200411);
\draw [line width=0.8pt] (2.8250179213391977,0.5224913288200411)-- (2.8250179213391977,0.7256767829345606);
\draw [line width=0.8pt] (2.943013355541281,0.810866802846994)-- (2.8250179213391946,0.8108668028469942);
\draw [line width=0.8pt] (2.8250179213391946,0.8108668028469942)-- (2.825017921339195,0.6928713686449077);
\draw [line width=0.8pt] (2.825017921339195,0.6928713686449077)-- (2.943013355541281,0.692871368644908);
\draw [line width=0.8pt] (2.943013355541281,0.692871368644908)-- (2.943013355541281,0.810866802846994);
\draw [line width=0.8pt] (2.943013355541281,0.692871368644908)-- (2.825017921339195,0.6928713686449077);
\draw [line width=0.8pt] (2.825017921339195,0.6928713686449077)-- (2.8250179213391955,0.5748759344428218);
\draw [line width=0.8pt] (2.8250179213391955,0.5748759344428218)-- (2.943013355541281,0.5748759344428223);
\draw [line width=0.8pt] (2.943013355541281,0.5748759344428223)-- (2.943013355541281,0.692871368644908);
\draw [line width=0.8pt] (3.2438189126237993,0.6280566799665583)-- (3.1418629537083134,0.6280566799665585);
\draw [line width=0.8pt] (3.1418629537083134,0.6280566799665585)-- (3.1418629537083134,0.5261007210510723);
\draw [line width=0.8pt] (3.1418629537083134,0.5261007210510723)-- (3.2438189126237993,0.5261007210510723);
\draw [line width=0.8pt] (3.2438189126237993,0.5261007210510723)-- (3.2438189126237993,0.6280566799665583);
\draw [line width=0.8pt] (3.1418629537083134,0.6280566799665585)-- (3.0460861438180094,0.6280566799665586);
\draw [line width=0.8pt] (3.0460861438180094,0.6280566799665586)-- (3.0460861438180094,0.5322798700762547);
\draw [line width=0.8pt] (3.0460861438180094,0.5322798700762547)-- (3.1418629537083134,0.5322798700762545);
\draw [line width=0.8pt] (3.1418629537083134,0.5322798700762545)-- (3.1418629537083134,0.6280566799665585);
\draw [line width=0.8pt] (3.0460861438180094,0.6280566799665586)-- (2.943013355541281,0.6280566799665588);
\draw [line width=0.8pt] (2.943013355541281,0.6280566799665588)-- (2.943013355541281,0.5249838916898304);
\draw [line width=0.8pt] (2.943013355541281,0.5249838916898304)-- (3.0460861438180094,0.5249838916898303);
\draw [line width=0.8pt] (3.0460861438180094,0.5249838916898303)-- (3.0460861438180094,0.6280566799665586);
\draw [line width=0.8pt] (2.943013355541281,0.5748759344428223)-- (2.8859463279440996,0.574875934442822);
\draw [line width=0.8pt] (2.8859463279440996,0.574875934442822)-- (2.8859463279440996,0.5178089068456407);
\draw [line width=0.8pt] (2.8859463279440996,0.5178089068456407)-- (2.943013355541281,0.5178089068456407);
\draw [line width=0.8pt] (2.943013355541281,0.5178089068456407)-- (2.943013355541281,0.5748759344428223);
\draw [line width=0.8pt] (2.8859463279440996,0.574875934442822)-- (2.8250179213391955,0.5748759344428218);
\draw [line width=0.8pt] (2.8250179213391955,0.5748759344428218)-- (2.8250179213391955,0.5139475278379178);
\draw [line width=0.8pt] (2.8250179213391955,0.5139475278379178)-- (2.8859463279440996,0.5139475278379179);
\draw [line width=0.8pt] (2.8859463279440996,0.5139475278379179)-- (2.8859463279440996,0.574875934442822);
\draw [line width=0.8pt] (3.5419723094289943,0.6307088402438834)-- (3.4390755109878297,0.6307088402438821);
\draw [line width=0.8pt] (3.4390755109878297,0.6307088402438821)-- (3.4390755109878315,0.5278120418027176);
\draw [line width=0.8pt] (3.4390755109878315,0.5278120418027176)-- (3.5419723094289957,0.527812041802719);
\draw [line width=0.8pt] (3.5419723094289957,0.527812041802719)-- (3.5419723094289943,0.6307088402438834);
\draw [line width=0.8pt] (3.4390755109878297,0.6307088402438821)-- (3.340616716797701,0.6307088402438822);
\draw [line width=0.8pt] (3.340616716797701,0.6307088402438822)-- (3.340616716797701,0.5322500460537534);
\draw [line width=0.8pt] (3.340616716797701,0.5322500460537534)-- (3.4390755109878297,0.5322500460537534);
\draw [line width=0.8pt] (3.4390755109878297,0.5322500460537534)-- (3.4390755109878297,0.6307088402438821);
\draw [line width=0.8pt] (3.340616716797701,0.6307088402438822)-- (3.2438189126238,0.6307088402438836);
\draw [line width=0.8pt] (3.2438189126238,0.6307088402438836)-- (3.243818912623798,0.5339110360699829);
\draw [line width=0.8pt] (3.243818912623798,0.5339110360699829)-- (3.3406167167976992,0.533911036069981);
\draw [line width=0.8pt] (3.3406167167976992,0.533911036069981)-- (3.340616716797701,0.6307088402438822);
\draw [line width=0.8pt] (4.,0.678104461654905)-- (3.8641663866677303,0.6781044616549047);
\draw [line width=0.8pt] (3.8641663866677303,0.6781044616549047)-- (3.864166386667731,0.5422708483226394);
\draw [line width=0.8pt] (3.864166386667731,0.5422708483226394)-- (4.,0.5422708483226398);
\draw [line width=0.8pt] (4.,0.5422708483226398)-- (4.,0.678104461654905);
\draw [line width=0.8pt] (3.8641663866677303,0.6781044616549047)-- (3.745157763543514,0.6781044616549048);
\draw [line width=0.8pt] (3.745157763543514,0.6781044616549048)-- (3.745157763543514,0.5590958385306887);
\draw [line width=0.8pt] (3.745157763543514,0.5590958385306887)-- (3.8641663866677303,0.5590958385306886);
\draw [line width=0.8pt] (3.8641663866677303,0.5590958385306886)-- (3.8641663866677303,0.6781044616549047);
\draw [line width=0.8pt] (3.745157763543515,0.7297612439968706)-- (3.644100369671894,0.7297612439968708);
\draw [line width=0.8pt] (3.644100369671894,0.7297612439968708)-- (3.644100369671894,0.6287038501252499);
\draw [line width=0.8pt] (3.644100369671894,0.6287038501252499)-- (3.745157763543515,0.6287038501252498);
\draw [line width=0.8pt] (3.745157763543515,0.6287038501252498)-- (3.745157763543515,0.7297612439968706);
\draw [line width=0.8pt] (3.644100369671894,0.7297612439968708)-- (3.5419723094289948,0.7312034070612733);
\draw [line width=0.8pt] (3.5419723094289948,0.7312034070612733)-- (3.540530146364592,0.6290753468183741);
\draw [line width=0.8pt] (3.540530146364592,0.6290753468183741)-- (3.642658206607491,0.6276331837539715);
\draw [line width=0.8pt] (3.642658206607491,0.6276331837539715)-- (3.644100369671894,0.7297612439968708);
\draw [line width=0.8pt] (3.642658206607491,0.6276331837539715)-- (3.5915095592080015,0.6283554601772997);
\draw [line width=0.8pt] (3.5915095592080015,0.6283554601772997)-- (3.5907872827846727,0.5772068127778103);
\draw [line width=0.8pt] (3.5907872827846727,0.5772068127778103)-- (3.6419359301841623,0.5764845363544817);
\draw [line width=0.8pt] (3.6419359301841623,0.5764845363544817)-- (3.642658206607491,0.6276331837539715);
\draw [line width=0.8pt] (3.5915095592080015,0.6283554601772997)-- (3.540530146364592,0.6290753468183741);
\draw [line width=0.8pt] (3.540530146364592,0.6290753468183741)-- (3.539810259723517,0.5780959339749644);
\draw [line width=0.8pt] (3.539810259723517,0.5780959339749644)-- (3.590789672566927,0.5773760473338898);
\draw [line width=0.8pt] (3.590789672566927,0.5773760473338898)-- (3.5915095592080015,0.6283554601772997);
\draw [line width=0.8pt] (3.745157763543515,0.6287038501252498)-- (3.6938550656876656,0.6287038501252499);
\draw [line width=0.8pt] (3.6938550656876656,0.6287038501252499)-- (3.693855065687666,0.5774011522694005);
\draw [line width=0.8pt] (3.693855065687666,0.5774011522694005)-- (3.745157763543515,0.5774011522694007);
\draw [line width=0.8pt] (3.745157763543515,0.5774011522694007)-- (3.745157763543515,0.6287038501252498);
\draw [line width=0.8pt] (3.6938550656876656,0.6287038501252499)-- (3.644100369671894,0.6287038501252499);
\draw [line width=0.8pt] (3.644100369671894,0.6287038501252499)-- (3.6441003696718934,0.5789491541094781);
\draw [line width=0.8pt] (3.6441003696718934,0.5789491541094781)-- (3.693855065687665,0.5789491541094779);
\draw [line width=0.8pt] (3.693855065687665,0.5789491541094779)-- (3.6938550656876656,0.6287038501252499);
\draw [line width=0.8pt] (3.8641663866677303,0.5590958385306886)-- (3.8107596057251354,0.5590958385306886);
\draw [line width=0.8pt] (3.8107596057251354,0.5590958385306886)-- (3.8107596057251354,0.5056890575880937);
\draw [line width=0.8pt] (3.8107596057251354,0.5056890575880937)-- (3.8641663866677303,0.5056890575880937);
\draw [line width=0.8pt] (3.8641663866677303,0.5056890575880937)-- (3.8641663866677303,0.5590958385306886);
\draw [line width=0.8pt] (3.8107596057251354,0.5590958385306886)-- (3.745157763543514,0.5590958385306887);
\draw [line width=0.8pt] (3.745157763543514,0.5590958385306887)-- (3.745157763543514,0.4934939963490672);
\draw [line width=0.8pt] (3.745157763543514,0.4934939963490672)-- (3.8107596057251354,0.4934939963490672);
\draw [line width=0.8pt] (3.8107596057251354,0.4934939963490672)-- (3.8107596057251354,0.5590958385306886);
\draw [line width=0.8pt] (3.745157763543515,0.5774011522694007)-- (3.693855065687666,0.5774011522694005);
\draw [line width=0.8pt] (3.693855065687666,0.5774011522694005)-- (3.693855065687667,0.5260984544135516);
\draw [line width=0.8pt] (3.693855065687667,0.5260984544135516)-- (3.7451577635435154,0.5260984544135523);
\draw [line width=0.8pt] (3.7451577635435154,0.5260984544135523)-- (3.745157763543515,0.5774011522694007);
\draw [line width=0.8pt] (4.,0.5422708483226398)-- (3.9545031021064716,0.5422708483226397);
\draw [line width=0.8pt] (3.9545031021064716,0.5422708483226397)-- (3.9545031021064707,0.49677395042911576);
\draw [line width=0.8pt] (3.9545031021064707,0.49677395042911576)-- (4.,0.4967739504291153);
\draw [line width=0.8pt] (4.,0.4967739504291153)-- (4.,0.5422708483226398);
\draw [line width=0.8pt] (3.9545031021064716,0.5422708483226397)-- (3.909014905504744,0.5422708483226395);
\draw [line width=0.8pt] (3.909014905504744,0.5422708483226395)-- (3.909014905504744,0.49678265172091196);
\draw [line width=0.8pt] (3.909014905504744,0.49678265172091196)-- (3.9545031021064716,0.49678265172091207);
\draw [line width=0.8pt] (3.9545031021064716,0.49678265172091207)-- (3.9545031021064716,0.5422708483226397);
\draw [line width=0.8pt] (3.909014905504744,0.5422708483226395)-- (3.864166386667731,0.5422708483226394);
\draw [line width=0.8pt] (3.864166386667731,0.5422708483226394)-- (3.8641663866677316,0.4974223294856266);
\draw [line width=0.8pt] (3.8641663866677316,0.4974223294856266)-- (3.909014905504744,0.4974223294856268);
\draw [line width=0.8pt] (3.909014905504744,0.4974223294856268)-- (3.909014905504744,0.5422708483226395);
\draw [line width=0.8pt] (3.693855065687665,0.5789491541094779)-- (3.6441003696718934,0.5789491541094781);
\draw [line width=0.8pt] (3.6441003696718934,0.5789491541094781)-- (3.644100369671894,0.5291944580937064);
\draw [line width=0.8pt] (3.644100369671894,0.5291944580937064)-- (3.693855065687665,0.5291944580937065);
\draw [line width=0.8pt] (3.693855065687665,0.5291944580937065)-- (3.693855065687665,0.5789491541094779);
\draw [line width=0.8pt] (3.6419359301841623,0.5764845363544817)-- (3.5907872827846727,0.5772068127778103);
\draw [line width=0.8pt] (3.5907872827846727,0.5772068127778103)-- (3.590065006361345,0.5260581653783207);
\draw [line width=0.8pt] (3.590065006361345,0.5260581653783207)-- (3.641213653760834,0.5253358889549925);
\draw [line width=0.8pt] (3.641213653760834,0.5253358889549925)-- (3.6419359301841623,0.5764845363544817);
\draw [line width=0.8pt] (3.5907872827846727,0.5772068127778103)-- (3.539810259723517,0.5780959339749644);
\draw [line width=0.8pt] (3.539810259723517,0.5780959339749644)-- (3.538921138526362,0.5271189109138086);
\draw [line width=0.8pt] (3.538921138526362,0.5271189109138086)-- (3.5898981615875183,0.5262297897166541);
\draw [line width=0.8pt] (3.5898981615875183,0.5262297897166541)-- (3.5907872827846727,0.5772068127778103);
\draw [line width=0.8pt,color=black] (2.,6.)-- (0.9245010396266959,3.0754989603733036);
\draw [line width=0.8pt,color=black] (2.,6.)-- (2.924501039626696,2.924501039626697);
\draw [line width=0.8pt,color=black] (0.9245010396266959,3.0754989603733036)-- (0.656731637597207,1.4942662831494);
\draw [line width=0.8pt,color=black] (0.9245010396266959,3.0754989603733036)-- (1.5812326772239031,1.8832285187171185);
\draw [line width=0.8pt,color=black] (1.5812326772239031,1.8832285187171185)-- (1.5812326772239031,1.347689714658141);
\draw [line width=0.8pt,color=black] (1.5812326772239031,1.347689714658141)-- (1.4648801457878167,0.92850344203525);
\draw [line width=0.8pt,color=black] (1.5812326772239031,1.347689714658141)-- (1.7326495478173052,0.9635677811925659);
\draw [line width=0.8pt,color=black] (0.656731637597207,1.4942662831494)-- (0.16425309775630234,0.6732815477958899);
\draw [line width=0.8pt,color=black] (0.656731637597207,1.4942662831494)-- (0.4891102281809157,0.6769306128838825);
\draw [line width=0.8pt,color=black] (0.656731637597207,1.4942662831494)-- (0.8135959268372979,0.6736529795641208);
\draw [line width=0.8pt,color=black] (0.656731637597207,1.4942662831494)-- (1.1454704340098918,0.6695418043676704);
\draw [line width=0.8pt,color=black] (2.924501039626696,2.924501039626697)-- (2.235417273239036,1.4625868852677506);
\draw [line width=0.8pt,color=black] (2.924501039626696,2.924501039626697)-- (3.0819023883268373,1.3889321581512355);
\draw [line width=0.8pt,color=black] (2.924501039626696,2.924501039626697)-- (3.770986154714497,1.619988233967892);
\draw [line width=0.8pt,color=black] (3.770986154714497,1.619988233967892)-- (3.770986154714497,1.1619605433968896);
\draw [line width=0.8pt,color=black] (1.4648801457878167,0.92850344203525)-- (1.4013396432940843,0.6892102033421779);
\draw [line width=0.8pt,color=black] (1.4013396432940843,0.6892102033421779)-- (1.3618750343233677,0.5529220761135546);
\draw [line width=0.8pt,color=black] (1.4013396432940843,0.6892102033421779)-- (1.4497514024230376,0.5618692262717914);
\draw [line width=0.8pt,color=black] (1.4648801457878167,0.92850344203525)-- (1.5527565138874866,0.7135460689481155);
\draw [line width=0.8pt,color=black] (1.5527565138874866,0.7135460689481155)-- (1.5527565138874866,0.5864650639606508);
\draw [line width=0.8pt,color=black] (1.7326495478173052,0.9635677811925659)-- (1.7326495478173052,0.7308627183203932);
\draw [line width=0.8pt,color=black] (1.7326495478173052,0.7308627183203932)-- (1.6724233927019152,0.5583838105636105);
\draw [line width=0.8pt,color=black] (1.7326495478173052,0.7308627183203932)-- (1.7887759241380015,0.554284031768917);
\draw [line width=0.8pt,color=black] (2.235417273239036,1.4625868852677506)-- (2.0309701129296442,0.8942036576058552);
\draw [line width=0.8pt,color=black] (2.235417273239036,1.4625868852677506)-- (2.4173853069152877,0.871724530972716);
\draw [line width=0.8pt,color=black] (2.0309701129296442,0.8942036576058552)-- (1.9421935623364832,0.6190441408465115);
\draw [line width=0.8pt,color=black] (2.0309701129296442,0.8942036576058552)-- (2.124161596012736,0.6234590733364421);
\draw [line width=0.8pt,color=black] (2.4173853069152877,0.871724530972716)-- (2.2818860108224763,0.5983295064467444);
\draw [line width=0.8pt,color=black] (2.4173853069152877,0.871724530972716)-- (2.4204266524062668,0.5976845932961146);
\draw [line width=0.8pt,color=black] (2.4173853069152877,0.871724530972716)-- (2.5559259484990777,0.6013708519377239);
\draw [line width=0.8pt,color=black] (3.0819023883268373,1.3889321581512355)-- (2.7234251942819374,0.8272695099918191);
\draw [line width=0.8pt,color=black] (3.0819023883268373,1.3889321581512355)-- (2.8840156384402382,0.869864519948035);
\draw [line width=0.8pt,color=black] (3.0819023883268373,1.3889321581512355)-- (3.09341613408254,0.778459458507818);
\draw [line width=0.8pt,color=black] (3.0819023883268373,1.3889321581512355)-- (3.3928956110263973,0.7797855386464796);
\draw [line width=0.8pt,color=black] (2.7234251942819374,0.8272695099918191)-- (2.723425194281938,0.6240840558773009);
\draw [line width=0.8pt,color=black] (2.8840156384402382,0.869864519948035)-- (2.884015638440238,0.7518690857459509);
\draw [line width=0.8pt,color=black] (2.884015638440238,0.7518690857459509)-- (2.8840156384402382,0.6338736515438649);
\draw [line width=0.8pt,color=black] (2.8840156384402382,0.6338736515438649)-- (2.8554821246416475,0.5444117311403699);
\draw [line width=0.8pt,color=black] (2.8840156384402382,0.6338736515438649)-- (2.9144798417426903,0.5463424206442314);
\draw [line width=0.8pt,color=black] (3.09341613408254,0.778459458507818)-- (2.994549749679645,0.5765202858281945);
\draw [line width=0.8pt,color=black] (3.09341613408254,0.778459458507818)-- (3.0939745487631614,0.5801682750214066);
\draw [line width=0.8pt,color=black] (3.09341613408254,0.778459458507818)-- (3.1928409331660563,0.5770787005088154);
\draw [line width=0.8pt,color=black] (3.3928956110263973,0.7797855386464796)-- (3.2922178147107495,0.5823099381569325);
\draw [line width=0.8pt,color=black] (3.3928956110263973,0.7797855386464796)-- (3.3898461138927654,0.5814794431488177);
\draw [line width=0.8pt,color=black] (3.3928956110263973,0.7797855386464796)-- (3.4905239102084127,0.5792604410233004);
\draw [line width=0.8pt,color=black] (3.770986154714497,1.1619605433968896)-- (3.6435650364862546,0.8313539710541298);
\draw [line width=0.8pt,color=black] (3.770986154714497,1.1619605433968896)-- (3.872578881771756,0.8055255798831467);
\draw [line width=0.8pt,color=black] (3.6435650364862546,0.8313539710541298)-- (3.592315258018243,0.6794182954076224);
\draw [line width=0.8pt,color=black] (3.6435650364862546,0.8313539710541298)-- (3.694629066607704,0.6792325470610602);
\draw [line width=0.8pt,color=black] (3.872578881771756,0.8055255798831467)-- (3.804662075105622,0.6186001500927967);
\draw [line width=0.8pt,color=black] (3.872578881771756,0.8055255798831467)-- (3.932083193333863,0.6101876549887721);
\draw [line width=0.8pt,color=black] (3.932083193333863,0.6101876549887721)-- (3.9772515510532336,0.5195223993758777);
\draw [line width=0.8pt,color=black] (3.932083193333863,0.6101876549887721)-- (3.931759003805608,0.5195267500217758);
\draw [line width=0.8pt,color=black] (3.932083193333863,0.6101876549887721)-- (3.886590646086238,0.519846588904133);
\draw [line width=0.8pt,color=black] (3.804662075105622,0.6186001500927967)-- (3.837462996196433,0.5323924480593911);
\draw [line width=0.8pt,color=black] (3.804662075105622,0.6186001500927967)-- (3.7779586846343247,0.5262949174398779);
\draw [line width=0.8pt,color=black] (3.694629066607704,0.6792325470610602)-- (3.7195064146155903,0.6030525011973251);
\draw [line width=0.8pt,color=black] (3.7195064146155903,0.6030525011973251)-- (3.7195064146155907,0.5517498033414762);
\draw [line width=0.8pt,color=black] (3.694629066607704,0.6792325470610602)-- (3.6689777176797795,0.603826502117364);
\draw [line width=0.8pt,color=black] (3.6689777176797795,0.603826502117364)-- (3.6689777176797795,0.5540718061015921);
\draw [line width=0.8pt,color=black] (3.592315258018243,0.6794182954076224)-- (3.5656599094657593,0.603225697076132);
\draw [line width=0.8pt,color=black] (3.592315258018243,0.6794182954076224)-- (3.616722744696082,0.6024199982658909);
\draw [line width=0.8pt,color=black] (3.616722744696082,0.6024199982658909)-- (3.6160004682727536,0.5512713508664012);
\draw [line width=0.8pt,color=black] (3.5656599094657593,0.603225697076132)-- (3.5648542106555174,0.5521628618458094);
\draw [line width=0.8pt,color=black] (2.,8.)-- (2.,6.);

\draw [line width=2.pt,color=black] (4.,8.)-- (0.,8.);
\draw [line width=2.pt,color=black] (0.,8.)-- (0.004368420828023068,0.5076994075353054);
\draw [line width=2.pt,color=black] (0.004368420828023068,0.5076994075353054)-- (4.007847529220136,0.501168446510196);
\draw [line width=2.pt,color=black] (4.007847529220136,0.501168446510196)-- (4.,8.);
\draw [color=black](-0.002668709472009345,8.64) node[anchor=north west] {$I^\ast\mu(\omega)=\mu^A(\partial T)=c_2(A)$};
\draw [color=black](2,7.25) node[anchor=north west] {$\omega$};
\draw [color=black](-0.5050465391305167,4.349537230555931) node[anchor=north west] {$1$};

\draw (-0.03481734532263382,3.85) node[anchor=north west] {$Q_\alpha$};
\draw [color=black](1.02,5) node[anchor=north west] {$\alpha$};

\draw (-0.03481734532263382,2) node[anchor=north west] {$Q_\beta$};
\draw [color=black](0.3,2.7331243768413462) node[anchor=north west] {$\beta$};

\end{tikzpicture}
\columnbreak

%\end{minipage}
%\begin{minipage}[t]{0.5\textwidth}
\begin{thm}\label{main 2}
(i) Let $T=(V,E)$ be a rooted tree and $\mu^A$ be the equilibrium measure (for $p=2$) for a set $A\subseteq\partial T$.

 Then, there exists a square tiling $\lbrace Q_\alpha\rbrace_{\alpha\in E}$ of the rectangle $R=[0,c_2(A)]\times[0,1]$, 
 where the combinatorics of the tiling are prescribed by $T$ and the square $Q_\alpha$ has side of 
length $\ell(\alpha)=\mu^A(\partial T_{\alpha})$.

(ii) Conversely, suppose a rectangle $R=[0,c]\times[0,1]$, $c\leq 1$, is square-tiled by $\lbrace Q_\alpha\rbrace_\alpha$ with combinatorics given by a rooted tree $T$.

Then there exists an $F_\sigma$ subset $A$ of $\partial T$ such that the measure $\mu(\partial T_{\alpha})=\ell(\alpha)$, 
 is the equilibrium measure of $A$, and then $c_2(A)=c$.
\end{thm}
\end{multicols}
A consequence of this theorem is that the  different tilings having the same tree-combinatorics can be parametrized by a family of $F_\sigma$ subsets of $\partial T$. 

%\end{minipage}

Theorem \ref{main 2} is very much related to results in \cite{benjamini}, \cite{schramm}, \cite{georgakopoulos}. Benjamini and Schramm proved that the equilibrium measure of a planar graph's boundary is associated to a tiling of a 
cylinder by squares, and on trees, analytically, this is the content of \eqref{equilibrium equation}. The converse statement is not true: there are tilings of cylinders whose combinatorics are prescribed by a planar graph, in which the sizes of the tiles do not reflect the 
equilibrium measure of the graph's boundary; trees provide plenty of counterexamples. Theorem \ref{main 2} provides, in the special case of trees, the correct bijection between tilings and discrete potential theoretic objects. It would be very interesting having a generalization of this statement to the whole class of planar graphs.

Again in the linear case, Theorem \ref{main 1} is related to some beautiful theorems of Kai-Lai Chung 
(see \cite{chung1973}, \cite{chung2008}), interpreting equilibrium measures in terms of last exit times for stochastic 
processes. The results of Chung, and of Benjamini and Schramm, have a probabilistic statement, or proof. 
As we work in the nonlinear case $1<p<\infty$, we do not expect probabilistic methods to apply here.
Even in the linear case, however, it would be interesting having a converse to Benjamini and Schramm's theorem on planar graphs, and this might have
a probabilistic proof. See the monographs \cite{woess2009}, \cite{lyons} and \cite{soardi} for thorough introductions to the stochastic 
processes which are here relevant.

The paper is organized as follows. Section \ref{SecPotential} is devoted to present some preliminary results of Potential Analysis on trees. Section \ref{SecRescaling} contains the proof of (i) in Theorem \ref{main 1}, which follows quite easily from some rescaling properties we present. In Section \ref{SecFlows} we provide some useful reinterpretation of equation \eqref{equilibrium equation} in terms of edge functions and of measures. The proof of part (ii) of our characterization is more subtle and is given in Section \ref{SecSufficient}. In Section \ref{SecTiling} we discuss some new results on square tilings of rectangles which follow from Theorem \ref{main 1} with $p=2$. Finally, in Section \ref{SecBranching} we show how capacities can be expressed by a recursive formula involving branched continued fractions, and we exploit this fact to give another reformulation of Theorem \ref{main 1}.

It is a pleasure to acknowledge useful discussions with Davide Cordella and Nikolaos Chalmoukis.

\section{Potential Theory on the tree}\label{SecPotential}

%Observe that edges automatically get an orientation out of the order relation, and we can well identify $E$ with the subset of $V\times V$ made of points $(x,y)$ such that $x\sim y$ and $y>x$. In such a case we write $[x,y]$ to denote the edge connecting $x$ and $y$, while the geodesic segment $[y,x]$ is not considered to be an edge. However, we will often prefer to simply denote edges by Greek letters, specifying their defining vertices only at need: given $\alpha \in E$, $b(\alpha)$ and $e(\alpha)$ are its beginning and ending vertex, respectively, i.e. they are connected by $\alpha$ and $b(\alpha)<e(\alpha)$.

Let $T$ be a tree with root vertex $o$ and root edge $\omega$, as described in the Introduction. A vertex is said to be a \textit{leaf} if it is the endpoint of one edge only. Being mainly interested in infinite trees, we will assume that $T$ has no leaves except $o$; this is a simplification but not a restriction. Given $\alpha \in E$, we write $b(\alpha)$ and $e(\alpha)$ for its beginning and ending vertex, respectively, i.e., they are connected by $\alpha$ and $b(\alpha)<e(\alpha)$. We denote by $|\alpha|$ the \textit{level} of an edge $\alpha$, which is the number of edges preceding $\alpha$ in the geodesic to the root $\omega$. With this definition we have $|\omega|=0$. The level of a vertex is the level of the subsequent edge, $|b(\alpha)|=|\alpha|$. We define the \textit{sons} of an edge $\alpha$ as the elements of the set $s(\alpha)=\lbrace\beta\in E: \ \beta\sim\alpha, \ |\beta|=|\alpha|+1\rbrace$. Given an edge $\alpha$, we write $T_\alpha$ for the subtree rooted at $\alpha$ having as edge set $\lbrace \beta\in E: \beta\geq\alpha \rbrace$, and we call it the $\alpha$-\textit{tent}. The boundary $\partial T_\alpha$ of the $\alpha$-tent is exactly the set of labels corresponding to rays passing through $\alpha$. We put on $\partial T$ the topology generated by the boundary of tents, namely the one having as a basis $\lbrace\partial T_\alpha\rbrace_{\alpha\in E}$.

We say that $S$ is a subtree of $T$ if it is a tree having the same root as $T$ and having as vertex set a subset of $V$. Observe that subtrees are exactly those subsets that can be obtained subtracting from $T$ a countable union of tents.

It is clear that the standard counting metric $\#$ does not extend properly to the boundary of an infinite tree. We enodow $\overline{V}$ with a different metric: given two points $\xi,\eta\in \overline{V}$, we define their \textit{confluent} to be the vertex $\xi\wedge\eta=\max \lbrace x\in V: x\leq \xi, x\leq \eta\rbrace$. Then, we define the \textit{visual metric}
\begin{equation}\label{gromovmetric}
d(\xi,\eta)=e^{-|\xi\wedge\eta|}.
\end{equation}
The reader familiar with Gromov's theory of hyperbolic spaces can observe that for $\xi,\eta\in V$, expression (\ref{gromovmetric}) coincides with the Gromov product on $\left(V,\#\right)$, given by
\begin{equation*}
  (x,y)_o=\frac{1}{2}\left(\#(x,o)+\#(y,o)-\#(x,y)\right), \quad x,y\in V.
\end{equation*}
In fact, one can extend such a product to points in $\partial T$ by setting $(\xi,\eta)_o=\lim(x,y)_o$, where the limit is taken for $x\to\xi, y\to\eta$ along the rays labeled by $\xi$ and $\eta$. With this notation we have $d(\xi,\eta)=(\xi,\eta)_o$.
It is not hard to see that (\ref{gromovmetric}) defines a distance on $\overline{V}$, which in fact is an ultrametric, and that $\left(\overline{V},d\right)$ is a compact metric space (see, for example, \cite[p. 121]{soardi}). Moreover, the topology induced by this metric on $\partial T$ is the tent topology introduced before.

There is a one-to-one correspondence between compact sets in $\partial T$ and boundaries of subtrees, in the following sense.

\begin{prop}\label{prop:tree with prescribed compact boundary}
A set $K\subseteq\partial T$ is compact if and only if there exists a subtree $T_K\subseteq T$ such that $K=\partial T_K$.
\end{prop}
\begin{proof}
The fact that the boundary of a subtree is compact follows directly from the definition of subtree. Conversely, if $K$ is compact, consider the subtree $S\subseteq T$ having as edge set $E(S)=P(K):=\lbrace \alpha\in E: \ \alpha<\xi, \text{ for some } \xi\in\partial K\rbrace$. Clearly $K\subseteq\partial S$. On the other hand, if $\xi\in\partial S$, by definition of boundary of a tree, $P(\lbrace\xi\rbrace)\subseteq E(S)$. Now, suppose by contraddiction that $\xi\notin K$. Then, by compactness, there exists and an edge $\alpha\in E$ such that $\partial T_\alpha\bigcap K=\emptyset$ and $\xi\in\partial T_\alpha$. Hence, $\beta\notin P(K)$ for $\beta\in P(\lbrace\xi\rbrace)\subseteq E(S)$ with $|\beta|\geq|\alpha|$, leading to a contraddiction.
\end{proof}

We present now a Nonlinear Potential Theory on the tree $T$ which falls within the axiomatics developed in \S2.3-2.5 of the treatise of Adams-Hedberg \cite{adams}. We consider the compact metric space $(\overline{V},d)$ and we make $E$ into a measure space by endowing it with the counting measure. We introduce the kernel $k:\overline{V}\times E\to \mathbb{R}$, given by the indicator function $k(\xi,\alpha)=\mathbb{I}(\lbrace(\xi,\alpha): \ \alpha<\xi\rbrace)$. Observe that $k(\cdot,\alpha)$ is continuous on $\partial T$, since $\partial T_{\alpha}$ is open.

Given a function $ f: E \rightarrow \mathbb{R} $ and $p,p'\in(1,+\infty)$ such that $1/p+1/p'=1$, we define the \textit{$p-$potential} of $f$, $I_pf: \overline{V} \rightarrow \mathbb{R}\cup\{\pm \infty \} $, by
\begin{equation*}
  I_pf(\xi)=\sum_{\alpha\in E} k(\xi,\alpha)f(\alpha)|f(\alpha)|^{p'-2}=\sum_{E\ni\alpha < \xi }f(\alpha)|f(\alpha)|^{p'-2}.  
\end{equation*}
We fix the convention  $I_pf(o)=0$ and we ease the notation by simply writing $If$ in place of $I_2f$.

The \textit{co-potential} of a charge $\mu\in \mathcal{M}(\partial T)$ is defined as the edge function
\begin{equation*}
    I^\ast\mu(\alpha)=\int_{\overline{V}}k(\xi,\alpha)d\mu(\xi)=\mu\left(\partial T_\alpha\right), \quad \alpha\in E.
\end{equation*}
Observe that $\langle If,\mu\rangle_{L^2(\partial T)}=\langle f, I^\ast \mu\rangle_{\ell^2(E)}$.

To the charge $\mu$ one can then associate the \textit{nonlinear potential} $V_p\mu(\xi)=I_pI^\ast\mu (x)$, $\xi\in \overline{V}$. It is easily seen that for $p=2$ one recovers the logarithmic potential of $\mu$ on the metric space $(\overline{V},d)$,
\begin{equation*}
    V_2\mu(\xi)=\int_{\partial T}\log\frac{1}{d(\xi,\nu)}d\mu(\nu), \quad \xi\in \overline{V}.
\end{equation*}
The \textit{$p-$energy} of the charge $\mu$ is given by
\begin{equation*}
    \mathcal{E}_p(\mu)=\int_{\partial T}V_p\mu(\xi)d\mu(\xi)=\Vert I^\ast\mu\Vert_{p'}^{p'}.
\end{equation*}
Setting $\Omega_A=\lbrace f\in\ell^p(E): \ If\geq 1 \ \mbox{on} \ A \rbrace$, we define the $p$-\textit{capacity} of a set $A\subseteq \partial T$ as
\begin{equation*}
c_p(A)=\inf_{f \in \Omega_A}\Vert f\Vert_p^p.
\end{equation*}
A property holds $c_p$-a.e. on $A$ if it holds everywhere on $A$ but for a subset of zero $p$-capacity. 
A function $f$ is \textit{admissible} for $A\subseteq \partial T$ if $f\in\overline{\Omega_A}^{\ell^p}$. It can be proved that $\overline{\Omega_A}^{\ell^p}=\lbrace f\in\ell_+^p: \ If\geq 1 \ c_p-a.e. \ \mbox{on} \ A \rbrace$, see \cite{adams}, and that there exists a unique function $f^A\in \overline{\Omega_A}^{\ell^p}$, called $p$-\textit{equilibrium function} for $A$, such that
\begin{equation*}
c_p(A)=\min_{f \in \overline{\Omega_A}^{\ell^p}}\Vert f\Vert_p^p=\Vert f^A\Vert_p^p.
\end{equation*}
Moreover, for such a function it holds $If^A=1$ $c_p$-a.e. on $A$.

We call a point $\xi\in A$ \textit{irregular} for $A$ if $If^A(\xi)\neq 1$. Clearly the set of irregular points for any set $A$ has zero $p-$capacity.

As a set function, $p$-capacity is monotone, countably subadditive and regular from inside and outside, i.e.,
\begin{itemize}
\item[(in)] $\displaystyle c_p\big(\bigcup_n A_n\big)=\lim_n c_p(A_n)$, for any increasing sequence $(A_n)$ of arbitrary subsets of $\partial T$.
\item[(out)] $\displaystyle c_p\big(\bigcap_n(K_n)\big)=\lim_n c_p(K_n)$, for any decreasing sequence $(K_n)$ of compact subsets of $\partial T$.
\end{itemize}

A set $A\subseteq \partial T$ is \textit{capacitable} if $c_p(A)=\sup\lbrace c_p(K): \ K\subseteq A, \ K \ \text{compact}\rbrace=\inf\lbrace c_p(G): \ F\supseteq A, \ G \ \text{open}\rbrace$. From (in) and (out) follows that all Suslin sets, and in particular all Borel sets, are capacitable, see \cite{choquety}.
%Moreover, if $A=\bigcup_n A_n$, $(A_n)$ increasing, and $c_p(A)<\infty$, then $f^{A_n}$ converges strongly to $f^A$ in $\ell^p$.

Observe that without losing generality, in the definition of $p-$capacity the infimum can be taken over functions supported on the \textit{predecessor set} of $\overline{A}$, $P(\overline{A}):=\lbrace \alpha\in E: \ \alpha<\xi, \text{ for some } \xi\in \overline{A}\rbrace$. Namely, the capacity of a compact set $K$ only depends on the combinatorics of $P(K)$ and not on the rest of the tree, and if $T_K\subseteq T$ is the subtree having $K$ as a boundary (see Proposition \ref{prop:tree with prescribed compact boundary}), we have $c_p(K)=c_p(\partial T_K)$, where the right handside is intended as the capacity when $T_K$ is taken as the ambient space.

It is possible to give a dual definition of capacity in terms of measures on $\partial T$ rather than of functions on $E$.
The proof of the following theorem is a straightforward adaptation of an equivalent result \cite[Theorem 2.5.6]{adams} for capacities in $\mathbb{R}^n$ arising by regular enough kernels.

	\begin{thmA}
		Suppose that $ A \subseteq \partial T $ is a Suslin set. Then
		\begin{equation*}
		c_p(A)=\sup\{\mu(A)^p: \mu\geq 0, \,\,\, \supp(\mu) \subseteq A, \mathcal{E}_p(\mu) \leq 1 \}.
		\end{equation*}
		Moreover, there exists a unique positive Borel measure $ \mu^A $ supported in $ \overline{A} $, called the $ p-$equilibrium measure of $ A $, such that 
		\begin{equation*}
		\mu^A(\overline{A})=c_p(A)=\mathcal{E}_p(\mu^A),
		\end{equation*}and $I^*\mu^A(\alpha)^{p'-1}=f^A(\alpha)$, $\alpha\in E$.
	\end{thmA}

With this definition of capacity, it is clear that if a property $(P)$ holds $c_p-$a.e. on $A\subseteq \partial T$, then it holds $\mu-$a.e. on $A$ for every measure $\mu$ with $\mathcal{E}_p(\mu)<\infty$. To see this, let $B:=\lbrace \xi \in A: \ \neg (P) \rbrace$, so that $c_p(B)=0$. The measure $\displaystyle\nu:=\frac{\mu|_B}{\mathcal{E}_p(\mu)^{1/p^\prime}}$ satisfies $\mathcal{E}_p(\nu)\leq 1$. Then, $\nu(B)\leq c_p(B)=0$, from which it follows $\mu(B)=0$.

We conclude the section by showing that in $\partial T$ there exist compact subsets with arbitrary $p$-capacity.

\begin{prop}\label{every capacity}
Let $T^n$ be a homogeneus tree of degree $n$. For each real number $t\in[0,c_p(\partial T^n)]$ there exists a compact subset $K_t$ of the boundary $\partial T^n$ such that $c_p(K_t)=t$.
\end{prop}
\begin{proof}
Each edge $\alpha$ of $T^n$, except the root, can be given an index $i(\alpha)\in\lbrace 0,\dots,n-1\rbrace$ which distinguishes it from the other $n-1$ edges $\beta$ such that $b(\beta)=b(\alpha)$. We can define a map $\Lambda: \partial T^n\to [0,1]$ associating to each point $\xi=\lbrace \alpha_j\rbrace_{j=1}^\infty\in\partial T^n$ the number having expansion in base $n$ given by
\begin{equation*}
\Lambda(\xi)=\sum_{j=1}^\infty i(\alpha_j)n^{-j}.
\end{equation*}

The map $\Lambda$ is clearly onto but it fails to be injective because of the multiple representations of the rational numbers. 
Still, $\Lambda^{-1}(t)$ has at most two points. Moreover, $\Lambda$ is continuous, since
\begin{equation*}
|\Lambda(\xi)-\Lambda(\eta)|\leq(n-1)\sum_{j=|\xi\wedge\eta|+1}^\infty n^{-j}\approx n^{-|\xi\wedge\eta|}\longrightarrow 0, \quad \text{as} \ \rho(\xi,\eta)\to 0.
\end{equation*}
Now, consider the function $\varphi:[0,1]\longrightarrow\mathbb{R}$ given by $\varphi(t)=c_p(\Lambda^{-1}[0,t])$. 
This is an increasing map, and we know that $\varphi(0)=0$ and $\varphi(1)=c_p(\partial T^n)$. 
By the subadditivity of $c_p$, the continuity of $\Lambda$ and the regularity of capacity, we have
	\begin{equation*}
	\varphi(t+\varepsilon)-\varphi(t)\leq c_p\left(\Lambda^{-1}[t,t+\varepsilon]\right)\longrightarrow c_p(\Lambda^{-1}\lbrace t\rbrace).
	\end{equation*}
The right handside equals zero, since the preimage of a single point under $\Lambda$ is finite. By similar reasoning 
we estimate $\varphi(t)-\varphi(t-\epsilon)$.
It follows that $\varphi$ is continuous and 
$\varphi([0,1])=[0,c_p(\partial T^n)]$. The result is obtained picking $K_t=\Lambda^{-1}[0,\varphi^{-1}(t)]$.
\end{proof}

\section{Rescaling of capacities}\label{SecRescaling}
In this section we show that equilibrium measures rescale under changes of the root in a tree, 
in a sense that will be more clear soon. This is a point where trees behave much more simply than general planar graphs, and constitutes the key property to prove $(i)$ in Theorem \ref{main 1}. We introduce a subscript notation to indicate which is the root of the tree we are referring to. 
For example, give the rooted tree $T$ and some edge $\alpha\in E$, we write $I_{p,\alpha}$ for the $p-$potential operator acting on functions defined on the edges of the $\alpha-$tent $T_\alpha$, $I_{p,\alpha}f=I_p(\bigchi_{T_{\alpha}}f)$. Accordingly, $V_{p,\alpha}\mu=I_{p,\alpha}I^\ast\mu$, $\mathcal{E}_{p,\alpha}(\mu)=\sum_{\beta\geq\alpha}|I^\ast\mu(\beta)|^{p'}$ and $c_{p,\alpha}$ denotes the capacity when we consider $T_\alpha$ as ambient space. In the same fashion, if $A$ is a subset of $\partial T$, $A_{\alpha}=A\cap\partial T_{\alpha}$.

A question arises: if $A$ is a subset of $\partial T$ and $\mu$ its $p-$equilibrium measure, which is the $p-$equilibrium measure $\mu_\alpha$ for $A_{\alpha}$ in the tent $T_{\alpha}$? It is natural to expect that it is a rescaling of the measure $\mu$, i.e., $\mu_{\alpha}=k_{\alpha}\mu|_{A_{\alpha}}$ for some positive constant $k_{\alpha}$. In such a case, for $c_p-$a.e. $\xi$ in $A_{\alpha}$, we would have
\begin{equation*}
  1=V_{p,\alpha}\mu_\alpha(\xi)=k_{\alpha}^{p^\prime-1}V_{p,\alpha}\mu(\xi)=k_{\alpha}^{p^\prime-1}\Big(V_p\mu(\xi)-V_p\mu(b(\alpha))\Big)=k_{\alpha}^{p^\prime-1}\Big(1-V_p\mu(b(\alpha))\Big).  
\end{equation*}
It follows that the only possible candidate rescaling constant is
\begin{equation}\label{k}
k_{\alpha}=\Big(1-V_p\mu(b(\alpha))\Big)^{1-p}.
\end{equation}
We now prove that our ansatz is correct. This was already observed in \cite{arcozzi2014potential}.
\begin{prop}\label{prop:rescaling}
Let $\mu$ be the $p-$equilibrium measure for a set $A\subseteq \partial T$. Then,
\begin{equation*}
   \mu_{\alpha}:=k_{\alpha}\mu|_{A_{\alpha}}=\frac{\mu |_{\partial T_\alpha}}{\Big(1-V_p\mu(b(\alpha))\Big)^{p-1}}, 
\end{equation*}
is the $p-$equilibrium measure for $A_{\alpha}\subseteq \partial T_{\alpha}$.
\end{prop}
\begin{proof}
Let $\varphi_{\alpha}$ be the (unique) $p-$equilibrium function of $A_\alpha$ in $T_\alpha$. We want to prove that $\varphi_{\alpha}=I^\ast\mu_\alpha^{p'-1}$. Define the edge function,
\begin{equation*}
  f(\beta)= \begin{cases}
\Big(1-V_p\mu(b(\alpha))\Big)\varphi_{\alpha}(\beta) & \text{if }  \beta\geq\alpha\\
I^\ast\mu(\beta)^{p'-1} & \text{otherwise } .
\end{cases}  
\end{equation*}
Then $f$ is admissible for $A\subseteq \partial T$, since for $c_p-$a.e. $\xi \in A\setminus A_{\alpha}$, $If(\xi)= V_p\mu(\xi)=1$, while for $c_p-$a.e. $\xi\in A_{\alpha}$ it holds
\begin{equation*}
  If(\xi)= I_{\alpha}f(\xi)+If(b(\alpha))=\Big(1-V_p\mu(b(\alpha))\Big)I_{\alpha}\varphi_{\alpha}(\xi)+V_p\mu(b(\alpha))=1.
\end{equation*}
Clearly $I^\ast\mu^{p'-1}$ is an admissible function for $A_\alpha\subseteq\partial T_\alpha$, since $I_\alpha I^\ast\mu^{p'-1}\leq I I^\ast\mu^{p'-1}=1$, $c.p-$a.e. on $A$. Hence, it must be $\Vert \varphi_{\alpha}\Vert_{\ell^p(T_\alpha)}^p\leq \Vert I^\ast\mu^{p'-1}\Vert_{\ell^p(T_\alpha)}^p=\Vert I^\ast\mu\Vert_{\ell^{p'}(T_\alpha)}^{p'}$. It follows
\begin{equation*}
\begin{split}
\Vert f\Vert_p^p &=\Big(1-V_p\mu(b(\alpha))\Big)^p\Vert \varphi_{\alpha}\Vert_{\ell^p(T_\alpha)}^p+\Vert I^\ast\mu^{p'-1}\Vert_{\ell^p(T\setminus T_\alpha)}^p\\
&\leq\Vert I^\ast\mu\Vert_{\ell^{p'}(T_\alpha)}^{p'}+\Vert I^\ast\mu\Vert_{\ell^{p'}(T\setminus T_\alpha)}^{p'}=\mathcal{E}_p(\mu)=c_p(A).
\end{split}
\end{equation*}
Then, $f$ must be the equilibrium function for $A$, and by the uniqueness of the equilibrium function $f=I^\ast\mu^{p'-1}$. This implies $\varphi_{\alpha}=I^\ast\mu_\alpha^{p'-1}$.
\end{proof}
As an immediate consequence we have the following.
\begin{cor}\label{cor}
Let $A\subseteq \partial T$ be a set of positive capacity and $\alpha\in E$, $\alpha$ not the root edge of $T$, such that $A\subseteq \partial T_{\alpha}$. Then $c_{p,\alpha}(A)> c_p(A)$.
\end{cor}
\begin{proof}
Since $k_{\alpha}>1$, we have $c_{p,\alpha}(A)=\mu_{\alpha}(A)=k_{\alpha}\mu(A)>\mu(A)=c_p(A)$.
\end{proof}
Observe by passing that the above corollary is supported by \textit{visual intuition}, since we expect the capacity of a set to be larger if we look at the set from a closer point of view.

We can now give a necessary condition for a measure to be of equilibrium, thus proving $(i)$ in Theorem \ref{main 1}.

\begin{proof}[Proof of $(i)$ in Theorem \ref{main 1}]
We claim that for every $\alpha\in E$, $\mu$ solves the following equation:
\begin{equation}\label{eq:formula}
I^\ast\mu(\alpha)\Big(1-V_p\mu(b(\alpha))\Big)=\mathcal{E}_{p,\alpha}(\mu).
\end{equation}
Suppose the claim holds and let $g=V_p\mu$. It is easily seen that $\nabla g[x,y]=I^\ast\mu[x,y]^{p'-1}$ for every edge $[x,y]\in E$. It follows that $g$ solves the equilibrium equation \eqref{equilibrium equation}.

To prove the claim, let $\alpha\in E$. If $A_{\alpha}=\emptyset$, then by the topology of the tree follows that also $\overline{A}\cap\partial T_\alpha=\emptyset$. Since $\supp(\mu)\subseteq \overline{A}$, $I^\ast\mu(\alpha)=0$, and hence $I^\ast\mu(\beta)=0$ for all $\beta\geq\alpha$, and \eqref{eq:formula} trivially holds. Otherwise, on one hand we have
\begin{equation*}
    c_{p,\alpha}(A_{\alpha})=\mathcal{E}_{p,\alpha}(\mu_{\alpha})=\frac{\mathcal{E}_{p,\alpha}(\mu)}{\Big(1-V_p\mu(b(\alpha))\Big)^p},
\end{equation*}
and on the other,
\begin{equation*}
c_{p,\alpha}(A_{\alpha})=\mu_{\alpha}(A_{\alpha})=k_{\alpha}\mu(A_{\alpha})=
\frac{I^\ast\mu(\alpha)}{\Big(1-V_p\mu(b(\alpha))\Big)^{p-1}}.
\end{equation*}
Matching the two expressions we get the claim.
\end{proof}

\section{Flows on edges and $p-$harmonic functions}\label{SecFlows}
Observe that equation \eqref{eq:formula} can be regarded as a special instance of an equation for edge functions:
\begin{equation}\label{eq:edges}
f(\alpha)\Big(1-I_pf(b(\alpha))\Big)=\sum_{\beta\geq\alpha}|f(\beta)|^{p'}, \quad f:E\to\mathbb{R}.
\end{equation}
In the particular case when the function $f$ is the co-potential of a charge $\mu$ on the $\partial T$, the above equation reduces to \eqref{eq:formula}. It is clear that a necessary condition for an edge function $f$ to be the co-potential of a charge is to satisfy the following additivity condition
\begin{equation}\label{eq:flow tree}
f(\alpha)=\sum_{\beta\in s(\alpha)}f(\beta), \quad \text{for all } \alpha\in E.
\end{equation}
A function $f:E\to\mathbb{R}$ fulfilling \eqref{eq:flow tree} is said to be a \textit{flow} on $T$.

In \cite[Proposition 2.2]{chalmoukis2019some} a full characterization for edge functions which are co-potential of charges is given. The nonnegative case, however, is easy to rule out: an edge function $f$ is a nonnegative flow if and only if $f=I^\ast\mu$ for some nonnegative measure $\mu$ on $\partial T$.

It turns out that solutions of \eqref{eq:edges} with bounded $p-$potential are always flows, and hence they have a measure representation.

\begin{prop}\label{prop:main per funzioni}
Let $f:E\to\mathbb{R}$ be a solution of \eqref{eq:edges} such that $I_p f<1$ on $V$. Then $f=I^\ast\mu$ for some nonnegative measure $\mu$ on $\partial T$.
\end{prop}
\begin{proof}
First of all observe that for a function $f$ solving \eqref{eq:edges}, the condition $I_p f<1$ automatically implies that $f\geq 0$ on $E$. We want to show that $f$ satisfies \eqref{eq:flow tree} for every $\alpha\in E$. If $f(\alpha)=0$ for some edge $\alpha$, then the right handside of equation (\ref{eq:edges}) says that $f(\beta)= 0$ on all edges $\beta\geq \alpha$ and then clearly \eqref{eq:flow tree} holds in $\alpha$. Now, consider edges $\alpha$ such that $f(\alpha)\neq 0$. We have
\begin{equation}\label{eq:main per funzioni}
\begin{split}
    \sum_{\beta\ge\alpha}f(\beta)^{p'}-f(\alpha)^{p'}&=\sum_{\beta\in s(\alpha)}\sum_{\gamma\ge\beta}f(\beta)^{p'}=\sum_{\beta\in s(\alpha)}f(\beta)\Big(1-I_p f(b(\beta))\Big)\\
    &=\Big(1-I_p f(e(\alpha))\Big)\sum_{\beta\in s(\alpha)}f(\beta)\\
    &=\Big(1-I_p f(e(\alpha))\Big)\Big(-f(\alpha)+\sum_{\beta\in s(\alpha)}f(\beta)\Big)+\sum_{\beta\ge\alpha}f(\beta)^{p'}-f(\alpha)^{p'}.
    \end{split}
    \end{equation}
    Since $1-I_p f(e(\alpha))> 0$, for every $\alpha$, then $f$ is a flow, and the result follows.
    \end{proof}

A similar calculation as in the proof of Proposition \ref{prop:main per funzioni} can be used to show that, if $f$ is a (not necessarily positive) flow and \eqref{eq:formula} holds for $|\alpha|$ large, then it holds everywhere. We give here the explicit details.
\begin{prop}
Let $\alpha\in E$ and suppose that $f$ is a flow solving \eqref{eq:formula} for $\beta\in s(\alpha)$. Then $f$ solves \eqref{eq:formula} also in $\alpha$.
\end{prop}
\begin{proof}
Following the line of \eqref{eq:main per funzioni} and using the flow property of $f$ we easily obtain
\begin{equation*}
\begin{split}
\sum_{\beta\ge\alpha}|f(\beta)|^{p'}&=|f(\alpha)|^{p'}+\Big(1-I_p f(e(\alpha))\Big)\sum_{\beta\in s(\alpha)}f(\beta)\\
&=f(\alpha)\Big( f(\alpha)|f(\alpha)|^{p'-2}+1-I_p f(e(\alpha))\Big)=f(\alpha)\Big(1-I_p f(b(\alpha))\Big).\qedhere
\end{split}
\end{equation*}
\end{proof}

We end the section by pointing out that the flow condition \eqref{eq:flow tree} for edge functions has a neat counterpart for vertex functions. In fact, it is easy to observe \cite[Proposition 2.3]{chalmoukis2019some} that $f$ is a flow if and only if $I_pf$ is $p-$harmonic on $V\setminus\lbrace o\rbrace$. In this context, the $p$-\textit{Laplacian} of a function $g:V\to\mathbb{R}$ at a vertex $x$ is given by
\begin{equation*}
\Delta_p g (x):=\sum_{y\sim x}\Big(g(y)-g(x)\Big)\Big|g(y)-g(x)\Big|^{p-2},
\end{equation*}
and $g$ is $p$-\textit{harmonic} if $\Delta_p g \equiv 0$ on $V\setminus\lbrace o \rbrace$. From Proposition \ref{prop:main per funzioni} then immediately follows
\begin{cor}
Let $g:V\to\mathbb{R}$ be a solution of \eqref{equilibrium equation} with $\Vert g\Vert_\infty<1$. Then $g$ is $p-$harmonic. 
\end{cor}

\section{Proof of Theorem \ref{main 1}}\label{SecSufficient}
In this section, we prove $(ii)$ in Theorem \ref{main 1}, namely that for a measure with bounded nonlinear potential it is sufficient to solve equation \eqref{equilibrium equation} to be an equilibrium measure.

%What if we have a measure $\mu\in\mathcal{M}^+(\partial T)$ which satisfies equation (\ref{formula})? 
%We are proving that this is in fact sufficient to assure that it is the $p$-equilibrium measure of some set. The main problem depends on understanding the size of the irregular points.

%\begin{theorem}\label{suff}
%Let $\mu\in\mathcal{M}^+(\partial T)$ be a solution of (\ref{formula}). Then, there exists an $\mathcal{F}_{\sigma}$ set $E$ such that $\mu$ is its $p$-equilibrium measure.
%\end{theorem}

\begin{proof}[Proof of (ii) in Theorem \ref{main 1}]
Let $g:V\to\mathbb{R}$, with $\Vert g\Vert_\infty<1$, be a solution of \eqref{equilibrium equation}. Without loss of generality, normalize assuming $g(o)=0$ (the function $\tilde{g}=(g-g(o))/(1-g(o))$ solves \eqref{equilibrium equation}, is bounded by $1$ and vanishes at $o$).
Set $f(\alpha)=\nabla g(\alpha)^{p-1}$, $\alpha\in E$. Then $f$ solves \eqref{eq:edges} and $I_pf=g$ is bounded by $1$. It follows by Proposition \ref{prop:main per funzioni} that there exists a nonnegative Borel measure $\mu$ on $\partial T$ such that $f=I^\ast\mu$ and, consequently, $g=V_p\mu$.

Observe that $\mu$ has finite $p-$energy: since it solves \eqref{eq:formula} at every edge $\alpha$, by evaluating at $\alpha=\omega$ we get $\mathcal{E}_p(\mu)=\mu(\partial T)=V_p(\mu)(e(\omega))^{p-1}<1$. Moreover, $V_p\mu(\xi)=\lim_{x\to\xi}V_p\mu(x)\leq 1$. We show that indeed $V_p\mu=1$, $\mu-$a.e. on $\partial T$.

Let $S_N=\lbrace \alpha\in E: \ |\alpha|=N \rbrace$. To each $N\in \mathbb{N}$ we associate a piecewise-constant function $\Phi_N$ on the boundary, $\Phi_N(\xi)=(1-V_p\mu(b(\alpha)))$ for $\xi\in\partial T_{\alpha}$, $\alpha\in S_N$. Then we have
\begin{equation*}
\begin{split}
0&\leq\int_{\partial T} \Phi_N(\xi)d\mu(\xi)=\sum_{\alpha\in S_N}\int_{\partial T_{\alpha}} \Phi_N(\xi)d\mu(\xi)=\sum_{\alpha\in S_N}(1-V_p\mu(b(\alpha)))I^\ast\mu(\alpha)\\
&=\sum_{\alpha\in S_N}\mathcal{E}_{p,\alpha}(\mu)=\sum_{|\beta|\geq N}I^\ast\mu(\beta)^{p'}=\mathcal{E}_p(\mu)-\sum_{|\beta|< N}I^\ast\mu(\beta)^{p'}.
\end{split}
\end{equation*}
Since $\mathcal{E}_p(\mu)<\infty$, it follows that
\begin{equation*}
    \int_{\partial T} \Phi_N(\xi)d\mu(\xi) \longrightarrow 0, \text{ as } N\to +\infty.
\end{equation*}
Also, $\Phi_N(\xi)\searrow\Phi(\xi):=1-V_p\mu(\xi)\geq 0$ as $N\to +\infty$. Hence, by monotone convergence theorem, we obtain
\begin{equation*}
\int_{\partial T} \Phi(\xi)d\mu(\xi)=\int_{\partial T} (1-V_p\mu(\xi))d\mu(\xi)=0,
\end{equation*}
which gives $V_p\mu(\xi)=1$, $\mu-$ a.e. on $\partial T$.

Consider now the \textit{irregular points} for $\mu$, i.e., with the $\mu-$measure zero set 
\begin{equation*}
\mathcal{I}(\mu)=\lbrace \xi\in\supp(\mu): \ V_p\mu(\xi)<1\rbrace.
\end{equation*}
Let $B_{n}=\lbrace\xi\in\supp(\mu): \ V_p\mu(\xi)\leq 1-1/2^n\rbrace$. Clearly $B_n\subseteq B_{n+1}$ and $\mathcal{I}(\mu)=\bigcup_n B_n$. 
Fix $\varepsilon>0$, and choose a collection of edges $(\alpha_j^n)_{n\in \mathbb{N},j\in \mathcal{J}_n}$, such that $\lbrace \partial T_{j,n}\rbrace_{j\in \mathcal{J}_n}$ is an open cover of $B_n$, where $T_{j,n}$ is the $\alpha_j^n-$tent. Without loss of generally, we can assume that $(\alpha_j^n)_{n\in \mathbb{N},j\in \mathcal{J}_n}$ satisfies the following:
\begin{itemize}
\item[(i)] $B_n\subseteq \bigcup_{j\in \mathcal{J}_n} \partial T_{j,n}$
\item[(ii)] $T_{j,n}\bigcap T_{i,l}=\emptyset$, for $(j,n)\neq (i,l)$
\item[(iii)] $|\mathcal{J}_n|=m_n\in\mathbb{N}$
\item[(iv)] $\sum_{j\in \mathcal{J}_n} I^\ast\mu(\alpha_j^n)<\varepsilon/2^n$.
\end{itemize}
In fact, if the intersection in (ii) was not empty, one of the two would be contained in the other and can be replaced by it in the covering family. Moreover, all the sublevel sets $B_n$ are compact, since $V_p\mu$ is continuous. Hence, for each $n$, we can extract a finite subcover so that $B_n\subseteq \bigcup_{j\in \mathcal{J}_n} \partial T_{j,n}$ with $|\mathcal{J}_n|=m_n\in\mathbb{N}$, which is the finiteness condition on the index set in (iii). Finally, condition (iv) comes from the outer regularity of the measure $\mu$: since $0=\mu(B_n)=\inf \lbrace\mu(\partial T_{\alpha}): \alpha \in E,\partial T_{\alpha} \supseteq B_n  \rbrace$, there exist sequences $(\alpha_j^n)_j$ such that $\mu(\partial T_{j,n})\to 0$ and we can assume to properly extract each subcover from one of those.

Write $\partial T=F_{\varepsilon}\bigcup G_{\varepsilon}$, where $G_{\varepsilon}:=\bigcup_{j,n}\partial T_{j,n}$ and $F_{\varepsilon}=\partial T\setminus G_{\varepsilon}$. Observe that,
\begin{equation*}
\mu(\partial T)\geq \mu (F_{\varepsilon})=\mu(\partial T)-\mu (G_{\varepsilon})=\mu(\partial T)-\sum_{j,n}I^\ast\mu(\alpha_j^n)\geq\mu(\partial T)-\sum_n \varepsilon/2^n=\mu(\partial T)-\varepsilon.
\end{equation*}
Hence we have,
\begin{equation*}
\mu(\partial T)=\lim_{\varepsilon \to 0} \mu (F_{\varepsilon}).
\end{equation*}
Now, $V_p\mu\equiv 1$ on $A_\varepsilon:=\supp(\mu)\cap F_{\varepsilon}$, so that $V_p\mu$ is a $p$-admissible function for $A_\varepsilon$. Then, by definition of capacity we have
\begin{equation}\label{definition}
c_p(A_\varepsilon)\leq \mathcal{E}_p(\mu|_{F_{\varepsilon}})\leq\mathcal{E}_p(\mu)=\mu(\partial T).
\end{equation}
On the other hand, the measure $\displaystyle\nu^{\varepsilon}:=\frac{\mu|_{ F_{\varepsilon}}}{\mathcal{E}_p(\mu|_{F_{\varepsilon}})^{1/p^\prime}}$ is admissible for $A_\varepsilon$, since $\mathcal{E}_p(\nu^{\varepsilon})=1$.
By the dual definition of capacity it follows that
\begin{equation}\label{dual definition}
c_p(A_\varepsilon)\geq\nu^{\varepsilon}(A_\varepsilon)^p\geq\Big(\frac{\mu(F_{\varepsilon})}{\mu(\partial T)^{1/p^\prime}}\Big)^p.
\end{equation}
We are now ready to build up the candidate $F_{\sigma}$ set. Let $\lbrace\varepsilon(k)\rbrace_{k\in\mathbb{N}}$ be a sequence of positive numbers such that $\varepsilon(k)\searrow 0$ as $k\to +\infty$. Define $A:=\bigcup_k A_{\varepsilon(k)}$, which is clearly an $F_{\sigma}$ set. Observe that we can assume that the covers related to each choice of $\varepsilon$ are taken so that $G_{\varepsilon(k+1)}\subseteq G_{\varepsilon(k)}$. Therefore, $A_{\varepsilon(k)}\nearrow$ as $k\to +\infty$. It follows that
\begin{equation*}
\mu(A)=\lim_{k\to\infty}\mu(A_{\varepsilon(k)})=\mu(\partial T)-\lim_{k\to\infty}\mu(G_{\varepsilon(k)})\geq\mu(\partial T)-\lim_{k\to\infty}\sum_n\frac{\varepsilon(k)}{2^n}=\mu(\partial T),
\end{equation*}
while the reverse inequality is trivially true. Using this together with (\ref{definition}) and (\ref{dual definition}) and the regularity of the $p-$capacity, we obtain
\begin{equation*}
c_p(A)=\lim_{k\to +\infty}c_p(A_{\varepsilon(k)})=\mu(\partial T)=\mu(A).\qedhere
\end{equation*}
\end{proof}

It is clear from the proof that the situation is much easier for measures with no irregular points. 
%With the above general result in our hands, we can regard this evenience as a special case of it, which reads as follows.

\begin{cor}
Suppose that $\mu\in\mathcal{M}^+(\partial T)$ solves \eqref{eq:formula} and $V_p\mu\equiv 1$ on $\supp(\mu)$. Then, $\mu$ is the $p$-equilibrium measure of $\supp(\mu)$.
\end{cor}

\section{Infinite square tilings}\label{SecTiling}
In \cite{brooks} Brooks, Smith, Stone and Tutte considered the problem of tiling a rectangle with a finite number of squares and proved that to any finite 
connected planar graph $G$ can be associated such a tiling. The same graph can produce different tilings. Chosen any two vertices in $G$, they show how the 
associated tiling can be built in such a way to reflect this choice. In \cite{benjamini} Benjamini and Schramm extended 
this result to the infinite case, showing that infinite graphs can produce infinite tilings. Theorem \ref{main 1} can be reformulated, for $p=2$, in terms of square tilings of a rectangle: part $(i)$ of the theorem is essentially equivalent to the infinite tiling theorem by Benjamini and Schramm, in the special case when $G$ is a rooted tree $T$, hence providing a new and different proof of it. More interestingly, part $(ii)$ provides a converse result, in a sense that will be more clear once introduced the proper terminology.

Given a rectangle $R$, i.e, a closed planar region whose boundary is a rectangle, we say that a family of squares $Q=\lbrace Q_j\rbrace_j$ is a 
\textit{square tiling} of $R$ if $\inn(Q_i)\bigcap \inn(Q_j)=\emptyset$, for $i\neq j$, and $R=\overline{\bigcup_j Q_j}$. By rotation invariance of the problem we always think rectangles and squares to have sides parallel to the coordinates axes of $\mathbb{R}^2$, and we talk about upper (lower) and left (right) sides, as well as horizontal and vertical sides, in the obvious way. We write $B(j)$ and $E(j)$ for the upper and lower side of $Q_j$, respectively.

We say that \it the combinatorics of a family $Q$ of squares in the plane are prescribed by a tree $T$ \rm if the followings are true.
\begin{itemize}
\item[(1)] The squares in the family are indexed by the edges of the tree, $Q=\lbrace  Q_\alpha: \ \alpha\in E\rbrace$.
\item[(2)] $B(\alpha)\subseteq E(\beta)$ whenever $b(\alpha)=e(\beta)$.
\end{itemize}

\pagebreak

Theorem \ref{main 1} can be reformulated, for $p=2$, in terms of square tilings, leading to Theorem \ref{main 2}. We detail out the proof here.
\begin{proof}[Proof of Theorem \ref{main 2}]
(i) Given a tree $T$ with root edge $\omega$ and a set $A\subseteq \partial T$, let $\lbrace Q_\alpha\rbrace_{\alpha\in E}$ be a family of squares such that $Q_\alpha$ has side of 
length $\ell(\alpha)=\mu(\partial T_{\alpha})$, being $\mu=\mu^A$ the equilibrium measure of $A$. By the additivity of $\mu$ we can place the squares on the plane in such a way that,
\begin{equation*}
E(\beta)=\bigcup_{\beta\in s(\alpha)}B(\alpha).
\end{equation*}
With this choice, the combinatorics is prescribed by $T$. Moreover, it is clear that the interiors of the squares in the family are pairwise disjoint and that $\bigcup_{\alpha} Q_\alpha$ is both vertically and horizontally convex (its intersection with any vertical and horizontal line is either empty, or a point, or a line segment). Now, let $R$ be the rectangle having vertical sides of 
length $1$ and upper side coinciding with the upper side of $Q(\omega)$, so being of 
length $\mu(\partial T)=\mu(\overline{A})=c_2(A)$. Denote by $|\cdot|$ the area measure. Then,
\begin{equation*}
|R|=c_2(A)=I^*\mu(\omega)=\mathcal{E}_2(\mu)=\sum_{\alpha\in E}\mu(\partial T_\alpha)^2=\sum_{\alpha\in E}|Q_\alpha|=\big|\bigcup_{\alpha} Q_\alpha\big|.
\end{equation*}
It is then enough to show that the family of squares is contained in $R$ to prove that it is a tiling. It is clear that all the family $\lbrace Q_\alpha\rbrace$ lies in between the two vertical sides of $R$, and that the horizontal room is fully filled, by additivity of the measure. Moreover, as already observed in the proof of part (ii) of Theorem \ref{main 1}, for every $\xi\in \partial T$ it must be
\begin{equation*}
1\geq V_2\mu(\xi)=\sum_{\alpha<\xi}\mu(\partial T_{\beta})=\sum_{\alpha<\xi}\ell(\alpha).
\end{equation*}
It follows that $\bigcup_\alpha Q_\alpha\subseteq R$ and $\lbrace Q_\alpha\rbrace_\alpha$ is a tiling.\\
(ii) Let the rectangle $R$ be tiled according to the combinatorics of a tree $T$, as described above. Then for each $\alpha\in E$, we have:
\begin{equation*}
\sum_{\beta \geq\alpha}|Q(\beta)|=\ell(\alpha)\bigg(1-\sum_{\beta<\alpha}\ell(\beta)\bigg).
\end{equation*}
Hence, it is immediate that if we define a measure on $\partial T$ by $\mu(\partial T_{\alpha})=\ell(\alpha)$, it solves equation \eqref{eq:formula}. Then the harmonic function $g=V_2\mu$ is bounded by 1 on $V$ and solves \eqref{equilibrium equation}. By Theorem \ref{main 1}, $\mu$ must be the equilibrium measure of some $F_\sigma$ subset $A$ of $\partial T$.
\end{proof}

It might be interesting to informally discuss some features of the tiling, and its relation to the set $A$. The example below can provide a useful
illustration of what we are here saying in general terms.

For each $\xi$ in $\partial T$, let $\alpha_n$ be the only edge of level $n$ along the geodesic labelled by $\xi$, and choose a point
$x_{n+1}$ in $Q_{\alpha_n}$. Then, it is immediate that $\lim_{n\to\infty}x_n=:\lambda(\xi)$ exists in $R$, and that it does not depend on the precise location of the $x_n$'s in the corresponding squares. Let $\pi(\xi)$ be the orthogonal projection of $\lambda(\xi)$ onto the lower side of $R$, identified with $[0,c_2(A)]$.

%Some reductions can be performed to get more regular tilings induced by more regular sets.
\pagebreak

\begin{multicols}{2}
Let $A\subseteq \partial T$ and $\mu^A$ its equilibrium measure. The following facts are easy to check:
\begin{itemize}
\item[(i)] if we replace $T$ by its subtree obtained keeping only the edges $\alpha$ with $\mu^A(\partial T_{\alpha})>0$, then tiling does not have degenerate squares, and $c_2(A\cap\partial T_\alpha)>0$ for all edges $\alpha$.
 \item[(ii)] let $A^\prime$ be the set of regular points for $\mu^A$, i.e. $A^\prime=\{\xi\in\partial T:\ V_2\mu^A(\xi)=1\}$: then $\mu^{A^\prime}=\mu^A$, hence they induce the same tiling.
 %(from now on we replace $A$ by $A^\prime$);
 \item[(iii)] $\pi$ is injective but possibly at countably many points and surjective from $\partial T$ onto $[0,c_2(A)]$;
 \item[(iv)] $\mu^A(\pi^{-1}(A))=\ell(A)$ for all measurable sets $A\subseteq[0,c_2(E)]$ (where $\ell$ denotes length measure on $[0,c_2(A)]$);
 \item[(v)] let $\text{Ex}(A):=\{\xi\in\partial T:\ \pi(\xi)\ne\lambda(\xi)\}$: then, $\text{Ex}(A)=\partial T\setminus A^\prime$;
 %\item[(vi)] by passing to a subtree of $T$, we can always assume that $c_2(A\cap\partial T_\alpha)>0$ for all edges $\alpha$.
\end{itemize}

\columnbreak
\begin{figure}[H]
%\centering
\includegraphics[scale=0.75]{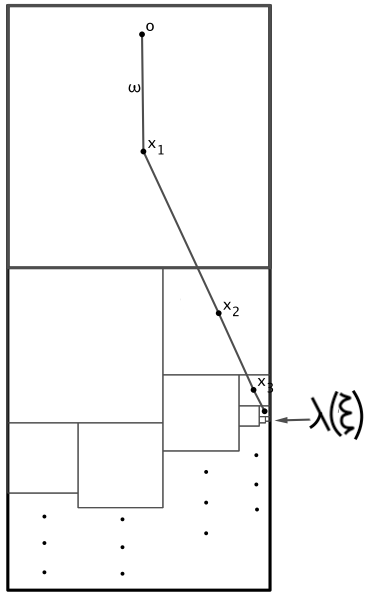}
\caption{For the boundary point $\xi$ identified by the geodesic $\lbrace x_j\rbrace_j$, $\lambda(\xi)$ does not lie on the bottom side of the $R$.}
\end{figure}
\end{multicols}

We can assume from now that the tree has been prune as in (i) and that $A=A^\prime$. Still, we see below that the combinatorics of the tree are not, by themselves, enough to determine a rectangle $R$ and a square tiling of it. They are, if we assume that the set $A$ in the Theorem \ref{main 2}
is \it closed\rm, but they are not in general. This is in striking contrast with the case of finite trees, or more generally graphs.
However, if $c_2(A)<c_2(\partial T)$, then a price has to be paid. In fact, in that case
\begin{equation*}
  c_2(\text{Ex}(A))\ge c_2(\partial T)-c_2(A)>0,  
\end{equation*}
i.e., the exceptional set $\text{Ex}(A)$ is rather large, although, clearly, $0=\mu^A(\text{Ex}(A))=\ell(\pi(\text{Ex}(A)))$.
To enlighten this phenomenon, we provide an example of a set which is densly spread out on the boundary having exceptional set of full capacity.

\begin{exmpl}[A regular set of dyadic combinatorics and arbitrarily small capacity with positive capacity in every subtree]
Let $\varepsilon>0$ be any small number, and $T=T_2$ a dyadic tree with edge root $\omega$. Let $n=n(\omega)$ be the number of steps one has to move to the left, starting from the root, before finding an edge $\alpha_{n}^\omega$ such that $c_p(\partial T_{n,\omega})\leq \varepsilon/2$, where $T_{n,\omega}$ denotes $\alpha_{n}^\omega-$tent. Let $\lbrace \alpha_1^\omega,\dots,\alpha_{n}^\omega\rbrace$ be the geodesic from the root to $\alpha_n^\omega$, and $\beta_j$ be the right brother of $\alpha_j^\omega$, i.e. the only edge with $b(\beta_j)=b(\alpha_j^\omega)$. In each subtree $T_{\beta_j}$, $j=1,\dots,n$, starting from the root $\beta_j$, move to the left, say $n(\beta_j)$ steps, until you find an edge $\alpha_{n(\beta_j)}^{\beta_j}$ such that $c_p(\partial T_{n(\beta_j),\beta_j})\leq \varepsilon/(2^2n)$. 
Then we iterate the process: $\lbrace \alpha_1^{\beta_j},\dots,\alpha_{n(\beta_j)}^{\beta_j}\rbrace$ is the geodesic from $\beta_j$ 
to $\alpha_{n(\beta_j)}^{\beta_j}$ and $\gamma_i=\gamma_i(j)$ the right brother of $\alpha_i^{\beta_j}$. In each subtree $T_{\gamma_i}$ we individuate 
as before, always moving to the left, subtrees with $c_p(\partial T_{n(\gamma_i),\gamma_i})\leq \varepsilon/(2^3n(\beta_j))$, and so on. Let
\begin{equation*}
 A_1= \partial T_{\alpha_{n}^\omega}, \quad A_2=\bigcup_{j=1}^n \partial T_{n(\beta_j),\beta_j}, \quad A_3=\bigcup_{j=1}^n\bigcup_{i=1}^{n(\beta_j)} \partial T_{n(\gamma_i)\gamma_i}, \quad \dots, \quad \text{and set} \quad A=\bigcup_k A_k.
\end{equation*}
By construction, for every $\alpha\in E$ the tree $T_\alpha$ contains a tent with boundary in $A$. Since tents have positive capacity (for example by the rescaling property of Proposition \ref{prop:rescaling}), it follows that $c_p(A\cap\partial T_\alpha)>0$ for every $\alpha\in E$. On the other hand,
\begin{equation*}
c_p(A)\leq\sum_k c_p(A_k)\leq \sum_k \frac{\varepsilon}{2^k}=\varepsilon.
\end{equation*}
For the regularity, observe that by construction for every point $\xi\in A$, there exists some edge $\alpha$ such that $\xi\in\partial T_\alpha\subseteq A$. Therefore, if $\mu$, $\mu^\alpha$ are the equilibrium measures for $A$ and $\partial T_\alpha$ respectively, we have
\begin{equation*}
V_p\mu(\xi)=V_{p,\alpha} \mu(\xi)+V_p \mu(b(\alpha))=\big(1-V_p\mu(b(\alpha)) \big)V_{p,\alpha} \mu^\alpha(\xi)+V_p\mu(b(\alpha))=1,
\end{equation*}
by the regularity of homogeneous trees.
\end{exmpl}

\begin{comment}
{\color{blue}
Moreover, given any real number $c\in(0,1)$, it is possible to build a tree with $c_2(\partial T)=c$. It follows that we can perform square tilings of rectangles with any ratio of sides.

PRIMA AVEVAMO UNA PROP PER DIMOSTRARE L'AFFERMAZIONE SOPRA, L'HO TOLTA E TOGLIEREI ANCHE L'AFFERMAZIONE PER NON APPESANTIRE L'ARTICOLO....POSSO RECUPERARLA QUANDO SCRIVERÒ L'ARTICOLETTO SULLA REGOLARITÀ DEI BORDI}
\end{comment}

\section{Branched continued fractions}\label{SecBranching}
Theorem \ref{main 1} can be reformulated in terms of \textit{branched continued fractions}. 
Besides adding further interesting structure to the class of equilibrium measures, this provides a recursive formula for concretely calculating capacity of sets. An accessible survey on branched continued fractions is in \cite{bodnar1999}.

\begin{prop}\label{cont frac}
Let $f:E\to\mathbb{R^+}$ be any non-negative function such that $I_pf<1$ on $V$, and consider the associated rescaled function defined by
\begin{equation}\label{sequence condition}
c(\alpha)=\frac{f(\alpha)}{\Big(1-I_pf(b(\alpha))\Big)^{p-1}}.
\end{equation}
Then, $f$ is the potential of a measure $\mu$ on $\partial T$ if and only if $c$ is defined, for each edge $\alpha$ which is not a leaf, by the following recursive formula
\begin{equation}\label{eq:precursive}
c(\alpha)=\frac{\displaystyle\sum_{\beta\in s(\alpha)}c(\beta)}{\left(1+\left(\displaystyle\sum_{\beta\in s(\alpha)}c(\beta)\right)^{p^\prime-1}\right)^{p-1}}.
\end{equation}

\end{prop}

\begin{proof}
By \eqref{sequence condition}, $f(\omega)=c(\omega)$. Denote by $\alpha^-$ the parent of $\alpha$, i.e. the only edge $\alpha^-$ such that $\alpha\in s(\alpha^-)$. For every $\alpha\neq\omega$, we have
\begin{equation*}
\begin{split}
f(\alpha)^{p'-1}&=c(\alpha)^{p^\prime-1}\Big(1-I_pf(b(\alpha))\Big)=c(\alpha)^{p^\prime-1}\left(1-I_pf(b(\alpha^-))-f(\alpha^-)^{p'-1}\right)\\
&=c(\alpha)^{p^\prime-1}\left(\frac{f(\alpha^-)^{p'-1}}{c(\alpha^-)^{p^\prime-1}} -f(\alpha^-)^{p'-1}\right)=c(\alpha)^{p^\prime-1}f(\alpha^-)^{p'-1}\frac{1-c(\alpha^-)^{p^\prime-1}}{c(\alpha^-)^{p^\prime-1}}.
\end{split}
\end{equation*}
Iterating we obtain,
\begin{equation}\label{eq:precursive measure}
f(\alpha)=c(\alpha)\prod_{\gamma<\alpha}\left(1-c(\gamma)^{p^\prime-1}\right)^{p-1}.
\end{equation}
Hence, for any chosen edge $\alpha$ which is not a leaf, it holds
\begin{equation*}
\sum_{\beta\in s(\alpha)}f(\beta)=\prod_{\gamma\leq\alpha}\left(1-c(\gamma)^{p^\prime-1}\right)^{p-1}\sum_{\beta\in s(\alpha)}c(\beta).
\end{equation*}
Now, $f$ is the co-potential of a measure if and only if the flow condition \eqref{eq:flow tree} holds. Namely, using \eqref{eq:precursive measure}, \eqref{eq:flow tree} holds if and only if
\begin{equation*}
c(\alpha)=\left(1-c(\alpha)^{p^\prime-1}\right)^{p-1}\sum_{\beta\in s(\alpha)}c(\beta),
\end{equation*}
which is equivalent to \eqref{eq:precursive}, as can be seen solving with respect to $c(\alpha)$.
\end{proof}

By the rescaling properties of equilibrium measure (Proposition \ref{prop:rescaling}), we have that if $\mu$ is the $p-$equilibrium measure for a set $A\subseteq\partial T$, then $c(\alpha)=c_{p,\alpha}(A_\alpha)$. 
This gives us an algorithm to calculate the capacity of a set in $\partial T$ in terms of successive tents capacities. Moreover, by relation \eqref{eq:precursive} we deduce that capacities can be expressed by means of 
branched continued fractions. For example, by \eqref{eq:precursive} we obtain the expression
\begin{equation*}
c_2(\partial T)=\frac{1}{1+\displaystyle\frac{1}{\displaystyle\sum_{\beta\in s(\omega)}\displaystyle\frac{1}{1+\displaystyle\sum_{\gamma\in s(\beta)}\frac{1}{1+\displaystyle\dots}}}} \quad.
\end{equation*}

In \cite[p. 57]{soardi} the same structure was observed for the \textit{total resistence} $R$ of an infinite tree without edges of degree 1. In particular, for such a class of trees, we obtain the relation
\begin{equation*}
c_2(\partial T)=\frac{1}{1+R}.
\end{equation*}

To end the section, we give a reformulation of Theorem \ref{main 1} which provides a characterization of equilibrium measures by means of an equation for capacities.

\begin{thm}
(i) Let $\mu$ the equilibrium measure for a set $A\subseteq \partial T$ and $f=I^\ast\mu$. Then the function $c=c_f$ given by \eqref{sequence condition} solves 
\begin{equation}\label{formula for capacities}
c(\alpha)\left(1-c(\alpha)^{p'-1}\right)=\sum_{\beta>\alpha}c(\beta)^{p'}\prod_{\alpha\leq\gamma<\beta}\left(1-c(\gamma)^{p'-1}\right)^p, \quad \alpha\in E.
\end{equation}
(ii) Let $\mu$ be a measure on $\partial T$ such that $V_p(\mu)<1$ on $V$. Set $f=I^\ast\mu$. If the function $c=c_f$ solves \eqref{formula for capacities}, then there exists an $\mathcal{F}_{\sigma}$ set $A\subseteq\partial T$ such that $\mu$ is its $p$-equilibrium measure.
\end{thm}
\begin{proof}
Given any measure $\mu$, setting $f=I^\ast \mu$ and $c=c_f$, by (\ref{sequence condition}) we have
\begin{equation*}
f(\alpha)\big(1-I_pf(b(\alpha))\big)=\frac{f(\alpha)^{p'}}{c(\alpha)^{p'-1}}, \quad \alpha\in E.
\end{equation*}
If $\mu$ is an equilibrium measure then \eqref{eq:formula} holds, and conversely if $\mu$ is such that $V_p(\mu)<1$ and solves \eqref{eq:formula} it is the equilibrium measure of some $\mathcal{F}_{\sigma}$ set. But \eqref{eq:formula} holds if and only if
\begin{equation*}
f(\alpha)^{p'}\left(1-c(\alpha)^{p'-1}\right)=c(\alpha)^{p'-1}\sum_{\beta>\alpha}f(\beta)^{p'}.
\end{equation*}
But by Proposition \ref{cont frac}, we know that $c(\alpha)$ solves \eqref{eq:precursive}, or equivalently, $f(\alpha)$ is defined by \eqref{eq:precursive measure}. Hence, sobstituting above we get
\begin{equation*}
c(\alpha)^{p'}\prod_{\gamma<\alpha}\left(1-c(\gamma)^{p^\prime-1}\right)^p\left(1-c(\alpha)^{p'-1}\right)=c(\alpha)^{p'-1}\sum_{\beta>\alpha}c(\beta)^{p'}\prod_{\gamma<\beta}\left(1-c(\gamma)^{p^\prime-1}\right)^p,
\end{equation*}
which is (\ref{formula for capacities}).
\end{proof}

\bibliography{equilibrium}
	\bibliographystyle{siam}

\end{document}